\documentclass[11pt,reqno]{amsart}

\usepackage{verbatim}
\usepackage{amsmath, amssymb, amsfonts, euscript, enumerate,latexsym}
\usepackage{amsthm} 
\usepackage{pxfonts}
\usepackage{fancyhdr}
\usepackage{mathrsfs}
\usepackage{mathtools}
\usepackage{epsfig,subfigure}
\usepackage{graphicx}
\usepackage{color,wrapfig}
\usepackage{txfonts}
\usepackage{hyperref}

\usepackage{etoolbox}
\makeatletter
\patchcmd{\@maketitle}
  {\ifx\@empty\@dedicatory}
  {\ifx\@empty\@date \else {\vskip3ex \centering\footnotesize\@date\par\vskip1ex}\fi
   \ifx\@empty\@dedicatory}
  {}{}
\patchcmd{\@adminfootnotes}
  {\ifx\@empty\@date\else \@footnotetext{\@setdate}\fi}
  {}{}{}
\makeatother 
 
\newcommand{\R}{{\mathbb R}}
\newcommand{\N}{{\mathbb N}}

\newcommand{\cA}{{\mathcal A}}
\newcommand{\cB}{{\mathcal B}}

\newcommand{\cD}{{\mathcal D}}

\newcommand{\cP}{{\mathcal P}}

\newcommand{\cX}{{\mathcal X}}

\newcommand{\sD}{{\mathscr D}}

\newcommand{\ld}{\lambda}

\newcommand{\ve}{\varepsilon}
\newcommand{\al}{\alpha}
\newcommand{\be}{\beta}
\newcommand{\p}{\partial}
\newcommand{\vp}{\varphi}

\newcommand{\supp}{\operatorname{supp}}

\newcommand{\loc}{\operatorname{loc}}
 
\newcommand{\D}{\nabla}

\newcommand{\La}{\Delta}

\newcommand{\Div}{\operatorname{div}}

\newtheorem{thm}{Theorem}[section]
\newtheorem{lemma}[thm]{Lemma}
\newtheorem{cor}[thm]{Corollary}
\newtheorem{remark}[thm]{Remark}
\newtheorem{prop}[thm]{Proposition}

\theoremstyle{definition}

\begin{document}
\title[Asymptotic large time behavior of singular solutions]{Asymptotic large time behavior of singular solutions of the fast diffusion equation}

\author[Kin Ming Hui]{Kin Ming Hui}
\address{Kin Ming Hui: 
Institute of Mathematics, Academia Sinica Taipei, Taiwan, R. O. C.}
\email{kmhui@gate.sinica.edu.tw}

\author[Soojung Kim]{Soojung Kim}
\address{Soojung Kim: 
Institute of Mathematics, Academia Sinica Taipei, Taiwan, R. O. C.}
\email{soojung26@gmail.com; soojung26@math.sinica.edu.tw}
\date{\today}

\begin{abstract} 
We study the asymptotic large time behavior of singular solutions of the fast diffusion equation $u_t=\Delta u^m$  in $({\mathbb R}^n\setminus\{0\})\times(0,\infty)$ in the subcritical case $0<m<\frac{n-2}{n}$, $n\ge3$. Firstly, we prove the existence of singular solution $u$ of the above equation that is trapped in between self-similar solutions of the form of $t^{-\alpha} f_i(t^{-\beta}x)$, $i=1,2$, with  initial value $u_0$ satisfying $A_1|x|^{-\gamma}\le u_0\le A_2|x|^{-\gamma}$ for some constants $A_2>A_1>0$ and $\frac{2}{1-m}<\gamma<\frac{n-2}{m}$, where $\beta:=\frac{1}{2-\gamma(1-m)}$, $\alpha:=\frac{2\be-1}{1-m},$ and the self-similar profile $f_i$ satisfies the elliptic equation 
$$
\Delta f^m+\alpha f+\beta x\cdot \nabla f=0\quad \mbox{in ${\mathbb R}^n\setminus\{0\}$}
$$
with $\lim_{|x|\to0}|x|^{\frac{ \alpha}{ \beta}}f_i(x)=A_i$ and  $\lim_{|x|\to\infty}|x|^{\frac{n-2}{m}}{f_i}(x)= D_{A_i}  $ for some constants $D_{A_i}>0$. When $\frac{2}{1-m}<\gamma<n$, under  an integrability condition on the initial value $u_0$ of the singular solution $u$, we prove that the rescaled function 
$$ 
\tilde u(y,\tau):= t^{\,\alpha} u(t^{\,\beta} y,t),\quad{  \tau:=\log t},
$$
converges to some self-similar profile $f$ as $\tau\to\infty$. 
\end{abstract}
 
\maketitle
\tableofcontents 

\section{Introduction}

\setcounter{equation}{0}
\setcounter{thm}{0}

We study solutions of the Cauchy problem of the fast diffusion equation 
\begin{equation}\label{eq-fde}
u_t=\La u^m
\end{equation}  
in $(\R^n\setminus\{0\})\times (0,\infty)$, which blow up at the origin $x=0$ for all time, in the  subcritical case   $0<m <\frac{n-2}{n}$, $n\geq3$. The equation \eqref{eq-fde}  is the well-known heat equation for $m = 1$, porous medium equation for $m > 1,$ and fast diffusion equation for $ 0 < m < 1$, respectively, that model diffusive processes of heat flows and gas flows in various media \cite{A}, \cite{DK}, \cite{V2}. When $m=\frac{n-2}{n+2}, $ $n\geq3,$ the equation \eqref{eq-fde}  also  arises in the study of the Yamabe flow  
equation
\begin{equation}\label{eq-Yamabe}
\frac{\partial g}{\partial t}=-Rg 
\end{equation}
on $\R^n$ where $R$ is the scalar curvature of the metric $g(x,t)$ at time $t$ \cite{DKS}, \cite{DS2}, \cite{PS}, \cite{Y}. 
In fact the metric $g=u^{\frac{4}{n+2}}dx^2$ on an open set $\Omega\subset\R^n$, $n\ge 3$, evolves by the Yamabe flow  \eqref{eq-Yamabe}  for $0<t<T$ if and only if $u$ is a solution of 
$$
u_t=\frac{n-1}{m}\La u^m \quad\mbox{ in }\Omega\times (0,T)\quad\mbox{ with }m=\frac{n-2}{n+2}. 
$$
There is an extensive literature  on  the existence, uniqueness, regularity  and asymptotic behavior of  solutions of \eqref{eq-fde} in the case $m\geq1$ and in the supercritical case $\frac{n-2}{n}<m<1$. In the subcritical case $0<m\leq\frac{n-2}{n}$, the properties of the solutions of \eqref{eq-fde} are quite different \cite{V1} and have been extensively  studied  in recent years by P. Daskalopoulos, J. King, M. del Pino, N. Sesum, M. S\'aez, \cite{DKS,DPS, DS1, DS2,PS}, S.Y. Hsu \cite{Hs1,Hs2,Hs3}, K.M. Hui  \cite{Hui1,Hui2,Hui3}, M. Fila, J.L.  Vazquez, M. Winkler, E. Yanagida \cite{FVWY,FW}, A. Blanchet, M. Bonforte, J. Dolbeault, G. Grillo, J.L. Vazquez \cite{BBDGV, BDGV}, etc. We also refer the readers to the survey paper \cite{A} 
and the books \cite{DK}, \cite{V2} 
on the recent results on \eqref{eq-fde}.  
  
In this paper we are concerned with solutions of \eqref{eq-fde} in $\left(\R^n\setminus\{0\}\right)\times (0,\infty)$ which blow up at the origin $x=0$ for all time in the subcritical case $0<m<\frac{n-2}{n}$, $n\geq 3$. More precisely, we will prove global existence of solution $u$  of    the fast diffusion equation 
\begin{equation}\label{eq-fde-global-except-0} 
\left\{\begin{aligned}
u_t=\La u^m\quad&\mbox{in }(\R^n\setminus\{0\})\times(0,\infty)\\
u(\cdot,0)=u_0\quad&\mbox{in }\R^n\setminus\{0\}\end{aligned}\right.
\end{equation}
which blows up at the origin $x=0$ for all time with initial value $u_0$ satisfying the growth condition
\begin{equation}\label{eq-fde-initial}
A_1|x|^{-\gamma}\leq u_0(x)\leq A_2|x|^{-\gamma}\quad\mbox{in $\R^n\setminus\{0\}$} 
\end{equation}
for some constants $A_2>A_1>0$ and $\frac{2}{1-m}< \gamma < \frac{n-2}{m}$ where  $n\geq3$ and  $0<m<\frac{n-2}{n}$. 

We will adapt the method  in  \cite{DS1, DKS,Hs1}, which {uses} integrability of the solution near the origin, to study the asymptotic large time behavior of the solution of \eqref{eq-fde-global-except-0} when  $\frac{2}{1-m}< \gamma < n $. In this case the solution $u$ of \eqref{eq-fde-global-except-0}  with  initial value  $u_0$ satisfying \eqref{eq-fde-initial}
is also a weak solution to the Cauchy problem for the fast diffusion equation
\begin{equation}\label{eq-fde-global}
\left\{
\begin{aligned}
u_t=\La u^m\quad&\mbox{in $\R^n\times(0,\infty)$}\\
u(\cdot,0)=u_0\quad&\mbox{in $\R^n$}.\end{aligned}
\right.
\end{equation} 
 
The study of  existence and  large time asymptotics of solutions  of \eqref{eq-fde-global-except-0}   satisfying \eqref{eq-fde-initial}  relies  on the study of the  self-similar solutions of \eqref{eq-fde-global-except-0}    which have initial value of the form $A|x|^{-\gamma}$ for some constants $A>0$ and $\frac{2}{1-m}<\gamma<\frac{n-2}{m}$. For any  $\frac{2}{1-m}<\gamma<\frac{n-2}{m}$, we consider a radially symmetric self-similar solution of  \eqref{eq-fde} of the form
\begin{equation*}
{U(x,t):=t^{-\alpha} f(t^{-\beta}x),\quad (x,t)\in  \R^n\ \times(0,\infty)}
\end{equation*} 
where 
\begin{equation}\label{eq-def-alpha-beta-gamma}
\beta:=\frac{1}{2-\gamma(1-m)}\quad \mbox{ and }\quad\al:=\frac{2\be-1}{1-m}.
\end{equation} 
Then $(m-1)\al+2\be=1$, $\al=\be \gamma$, and $U(x,t)$ is a solution of \eqref{eq-fde-global-except-0} with initial value 
$U_0(x)=A|x|^{-\gamma}$ if and only if $f$ is a radially symmetric solution of
\begin{equation}\label{eq-fde-scaled}
\La f^m+\alpha f+\beta x\cdot \D f=0,\quad f>0
\end{equation}
in $\R^n\setminus\{0\}$ with 
\begin{equation}\label{eq-sol-behavior-origin}
\lim_{|x|\to 0}|x|^{\frac{ \alpha}{ \beta}}f(x)=A
\end{equation}
where we recall that $ \gamma=\frac{\al}{\be}$. Note that since $ \gamma >\frac{2}{1-m}$, $\al<0$ and $\be<0$.     
Since the asymptotic large time behavior of solution of \eqref{eq-fde-global-except-0} is {usually} similar to the self-similar solution of
\eqref{eq-fde} we will first prove the following result in our paper. 
  
\begin{thm}[Existence of self-similar profile] \label{thm-existence-self-similar}
Let $n\geq3$, $0< m < \frac{n-2}{n}$, 
\begin{equation}\label{eq-alpha-beta-neg-condition}
\beta<0,\quad\rho_1>0,\quad\alpha:=\frac{2\beta-\rho_1}{1-m}\quad\mbox{ and }\quad\frac{2}{1-m}<\frac{\alpha}{\beta}<\frac{n-2}{m}.
\end{equation}
For any $A>0,$ there exists a unique radially symmetric solution $f$ of \eqref{eq-fde-scaled} in $\R^n\setminus\{0\}$, which satisfies  \eqref{eq-sol-behavior-origin} and 
\begin{equation}\label{eq-sol-behavior-infty}
\lim_{|x|\to\infty}|x|^{\frac{n-2}{m}}f(x)= D_A  
\end{equation}
for some constant $D_A>0$  depending on $A$. Moreover, 
\begin{equation}\label{eq-hess-f^m}
\La f^m=-\left(\alpha f+ \beta x\cdot \D f\right) < 0 \quad\mbox{ in $\R^n\setminus\{0\}$}. 
\end{equation}
\end{thm}
We will prove Theorem \ref{thm-existence-self-similar} in section \ref{sec-inversion} using an inversion method which {transforms} the above problem into an equivalent existence problem of the related inversion elliptic equation. Note that a heuristic proof of the existence of solution of \eqref{eq-fde-scaled} in $\R^n\setminus\{0\}$ satisfying  \eqref{eq-sol-behavior-origin} for the case $\frac{2}{1-m}<\frac{\alpha}{\beta}<n$ using phase-plane analysis is given in Chapter 5 of \cite{V2}. 

We will let $n\geq3$ and $0< m < \frac{n-2}{n}$ for the rest of the paper.
In the case when $0<\gamma<\frac{2}{1-m}$, it was proved in \cite[Theorem 1.2]{Hs3} that a rescaled limit of the global smooth   solution $u$ of  \eqref{eq-fde-global} with initial value  $u_0(x)\approx A|x|^{-\gamma}$ as $t$ tends to infinity  is  a  
radially symmetric  self-similar profile $f$ which  satisfies  \eqref{eq-fde-scaled} in $\R^n$  with $\al>0$ and $\be>0$ given in \eqref{eq-def-alpha-beta-gamma}, and 
\begin{equation*}
\lim_{|x|\to\infty}|x|^{ \gamma}f(x)=A.
\end{equation*}    
In the case $\gamma=\frac{2}{1-m},$  the self-similar shrinking    Barenblatt type solution    $\cB_k$ of \eqref{eq-fde} defined by 
$$
\cB_k(x,t):=\left(\frac{C^*(T-t)}{|x|^2+k(T-t)^{2\sigma^*}}\right)^{\frac{1}{1-m}},\quad \forall (x,t)\in\R^n\times(0,T),
$$ 
where $T>0$ and $k\geq0$ are free parameters and  
$$
C^*:=\frac{2m(n-2-nm)}{1-m},\quad\mbox{and}\quad \sigma ^*:=-\frac{1}{n-2-nm},
$$
which vanishes identically at time $T$ is  well known. In particular when $k=0,$ 
$$
\cB_0(x,t)=\left(\frac{C^*(T-t)}{|x|^2 }\right)^{\frac{1}{1-m}}
$$
remains  singular at the origin for all time $t<T$ with $\cB_0(x,0)= {(C^*T)}^{\frac{1}{1-m}}|x|^{-\frac{2}{1-m}}$ and $\cB_0(x,T)\equiv0.$   
For general initial value satisfying the condition $u_0\approx A|x|^{-\frac{2}{1-m}}$ for some constant $A>0$ as $|x|\to\infty$, asymptotic behavior of the solution of \eqref{eq-fde} in $\R^n\times (0,T)$ near the extinction time $T$ has been studied in \cite{DS1, DKS, Hui3}. 

For the case $\gamma=\frac{n-2}{m}$, $A|x|^{-\frac{n-2}{m}}$ is a particular solution of \eqref{eq-fde} in $(\R^n\setminus\{0\})\times (0,\infty)$. 

\begin{remark}\label{self-similar-soln-scaling-property-rmk}
Let $\alpha$, $\beta$ and $\rho_1$ satisfy \eqref{eq-alpha-beta-neg-condition} and let $f_1$ be the radially symmetric solution  of \eqref{eq-fde-scaled} in $\R^n\setminus\{0\}$ which satisfies \eqref{eq-sol-behavior-origin}  and \eqref{eq-sol-behavior-infty} with $A=1$ for some constant $D_1>0$ given by Theorem \ref{thm-existence-self-similar}. 
For any $\ld>0$, we define   
\begin{equation}\label{eq-def-scaled-f1-lambda}
f_{\ld}(x):= \ld^{\frac{2}{1-m}} f_1(\ld x).
\end{equation}
Then $f_\ld$  satisfies \eqref{eq-fde-scaled} in $\R^n\setminus\{0\}$ and  
\begin{equation}\label{eq-behavior-f1_lambda}
\left\{\begin{aligned}
\lim_{|x|\to0} |x|^{\frac{\alpha}{\beta}}f_\ld(x)&= \lim_{|x|\to0} \ld^{\frac{2}{1-m}-\frac{\alpha}{\beta}}(\ld |x|)^{\frac{\alpha}{\beta}}f_1(\ld x)=\ld^{\frac{2}{1-m}-\frac{\alpha}{\beta}},\\
\lim_{|x|\to\infty} |x|^{\frac{n-2}{m}}f_\ld(x)&= \lim_{|x|\to\infty} \ld^{\frac{2}{1-m}-\frac{n-2}{m}}(\ld |x|)^{\frac{n-2}{m}}f_1(\ld x)=\ld^{\frac{2}{1-m}-\frac{n-2}{m}}D_1 . 
\end{aligned}\right.
\end{equation} 
By the uniqueness result of Theorem \ref{thm-existence-self-similar} and the scaling property above, the solution $f$ of \eqref{eq-fde-scaled} in $\R^n\setminus\{0\}$ which satisfies \eqref{eq-sol-behavior-origin} {and \eqref{eq-sol-behavior-infty} for  given constants  $A>0$ and $D_A>0$} coincides with the rescaled function $f_\ld$  given by  \eqref{eq-def-scaled-f1-lambda} with $\ld=A^{1/(\frac{2}{1-m}-\frac{\alpha}{\beta})}$ and 
\begin{equation*} 
D_A=D_1A^{(\frac{2}{1-m}-\frac{n-2}{m})/(\frac{2}{1-m}-\frac{\alpha}{\beta})}.
\end{equation*}  
 Observed   by Remark  \ref{rmk-scaling-monotonicity} in section \ref{sec-self-similar profiles} for any $0\ne x\in\R^n$, $f_\ld(x)$ is a monotone decreasing function of $\ld>0.$ 
\end{remark}

Let $\frac{2}{1-m}<\gamma<\frac{n-2}{m}$, $\rho_1=1$, and  $\al$, $\be$  be given by \eqref{eq-def-alpha-beta-gamma}. Then the self-similar profile $f_\ld$ given by \eqref{eq-def-scaled-f1-lambda} yields a self-similar solution     
\begin{equation}\label{eq-def-U-lambda}
U_\ld(x,t):=t^{-\alpha}f_\ld(t^{-\beta}x)\quad\forall (x,t)\in \left(\R^n\setminus\{0\}\right)\times(0,\infty)
\end{equation} 
of \eqref{eq-fde-global-except-0} with initial value $U_{\ld, 0}(x)=\ld^{\frac{2}{1-m}-\gamma} |x|^{-\gamma}$ since $\al=\be\gamma$ and 
$$
\lim_{t\to0} U_\ld(x,t)=\lim_{|y|=t^{-\beta}|x|\to0} |x|^{-\gamma}|y|^{\gamma}f_\ld(y)=\ld^{\frac{2}{1-m}-\gamma}|x|^{-\gamma}\quad\forall x\not=0.
$$
When $\frac{2}{1-m}<\gamma < n ,$  by \eqref{eq-behavior-f1_lambda} $U_\ld\in C\left( [0,\infty); L^1_{\loc}(\R^n)\right)\cap C\left( (0,\infty); L^1(\R^n)\right)$ is a weak solution of  \eqref{eq-fde-global} with initial value $ U_{\ld,0}(x)=\ld^{\frac{2}{1-m}-\gamma}|x|^{-\gamma}\in L^1_{\loc}(\R^n).$ 
 
When $\frac{2}{1-m} <\gamma< \frac{n-2}{m}$, we will prove the existence  of  solution of \eqref{eq-fde-global-except-0} trapped in between two  self-similar  solutions $U_{\ld_i},  i=1,2$, $\lambda_1>\lambda_2>0$, of the form  \eqref{eq-def-U-lambda} with initial value $u_0$ satisfying 
\begin{equation}\label{eq-initial-trapped-particular-sols-existence}
A_1|x|^{- \gamma} \leq u_0(x)\leq  A_2|x|^{-\gamma},\quad\forall x\in\R^n\setminus\{0\}
\end{equation} 
where $A_i=\ld_i^{\frac{2}{1-m} -\gamma} $, $i=1,2$. We will also establish a weighted $L^1$-contraction theorem for such solutions. Since $|x|^{-\gamma}$ is not integrable in $\R^n$,  the  difference of any two initial values $u_0,v_0$, that satisfy 
\eqref{eq-initial-trapped-particular-sols-existence}  may not be integrable  in $\R^n$. So we need to introduce a weighted $L^1$-space in order to study the asymptotic large time behavior of the solution of \eqref{eq-fde-global-except-0} with initial value $u_0$ satisfying 
\eqref{eq-initial-trapped-particular-sols-existence}.

For any $\mu>0,$ we define the weighted $L^1$-space with weight $|x|^{-\mu}$    by 
\begin{equation*}
L^1(r^{-\mu}; \R^n):=\left\{ h: \int_{\R^n } |h(x)| |x|^{-\mu} dx<\infty\right\}
\end{equation*}
with  norm 
$$
\|h\|_{L^1(r^{-\mu}; \R^n)}=\int_{\R^n } |h(x)| |x|^{-\mu} dx.
$$  
Let us fix some constants that will be used later. Let 
\begin{equation}\label{eq-weighted-L1-contraction-mu}
\mu_1:=\max\left(0, \,n-\frac{\al}{\be}\right)\quad\mbox{and}\quad 
\mu_2:= n-2-\frac{m\al}{\be} . 
\end{equation}
Unless stated otherwise we will now assume that $\frac{2}{1-m}<\frac{\alpha}{\beta}<\frac{n-2}{m}$ for the rest of the paper. Then $0\le \mu_1<\mu_2<n-2$.

\begin{thm}[Weighted $L^1$-contraction]\label{thm-weighted-L1-contraction-new}
Let $n\geq3,$ $0<m <\frac{n-2}{n}$, and $\frac{2}{1-m}<\gamma<\frac{n-2}{m}$. Let $u$ and $v$ be solutions of \eqref{eq-fde-global-except-0} which satisfy 
\begin{equation}\label{eq-trapped-particular-sols-new}
U_{\ld_1}\leq u,v \leq U_{\ld_2}\quad\mbox{in $\left(\R^n\setminus\{0\}\right)\times(0,\infty)$}
\end{equation}
where $U_{\ld_i}$, $i=1,2$,  are given by \eqref{eq-def-U-lambda} with $\al$ and $\be$ given by \eqref{eq-def-alpha-beta-gamma} and  
$\ld_1>\ld_2>0 $. Assume that $|u_0 -v_0|\in L^1\left(r^{-\mu};\R^n\right)$ for some constant $\mu\in (\mu_1, \mu_2)$. Then 
\begin{equation}\label{eq-weighte-L1-contraction-new}
\int_{\R^n} |u -v|(x,t) |x|^{-\mu}\,dx\leq\int_{\R^n}|u_0 -v_0|(x) |x|^{-\mu}\,dx\quad\forall t>0
\end{equation}
and
\begin{equation}\label{eq-weighte-L1-contraction-positive-part-new}
\int_{\R^n} \left(u -v\right)_+(x,t)|x|^{-\mu}\,dx\leq\int_{\R^n}\left(u_0 -v_0\right)_+(x)|x|^{-\mu}\,dx\quad\forall t>0.
\end{equation}
\end{thm}

\begin{thm}[Existence]\label{thm-existence-sol-fde-singular}
Let  $n\geq3,$ $0< m < \frac{n-2}{n} $, and  $\frac{2}{1-m}<\gamma<\frac{n-2}{m}$. Let $u_0$  satisfy 
\eqref{eq-initial-trapped-particular-sols-existence} for some constants $A_2>A_1>0.$ Then there exists a unique solution $u$ of
\eqref{eq-fde-global-except-0} satisfying  
\begin{equation}\label{eq-existence-sols-trapped-self-similar}  
U_{\ld_1}\leq u  \leq U_{\ld_2}\quad\mbox{in $\left(\R^n\setminus\{0\}\right)\times(0,\infty)$}, 
\end{equation}
where $U_{\ld_i}$ for $i=1, 2$, are given by \eqref{eq-def-U-lambda} with $\al$, $\be$, given by \eqref{eq-def-alpha-beta-gamma}, and $\lambda_i:=A_i^{1/ (\frac{2}{1-m}-\gamma )}$  for  $i=1,2,$ respectively.  
Moreover
\begin{equation}\label{eq-AB}
u_t\leq \frac{u}{(1-m)\,t}\quad\mbox{in   $\left(\R^n\setminus\{0\}\right)\times(0,\infty)$}.
\end{equation}
\end{thm}
For any solution $u$ of \eqref{eq-fde} in $(\R^n\setminus\{0\})\times(0,\infty)$ and constants $\al,\be$ satisfying $\al=\frac{2\be-1}{1-m},$ we define the rescaled function $\tilde u$ by
\begin{equation}\label{def-fde-rescaled}
\tilde u(y,\tau):= t^{\,\al} u(t^{\,\beta} y,t),\quad{  \tau:=\log t}.
\end{equation}
Then $\tilde u$ satisfes
\begin{equation}\label{eq-fde-rescaled}
\tilde u_\tau=\La \tilde u^m+\alpha \tilde u+\beta y\cdot\D \tilde u
\end{equation}  
in $\left(\R^n\setminus\{0\}\right)\times (-\infty,\infty)$ in the classical sense since $(m-1)\al+2\be=1$. In particular
$$
\tilde U_\ld(y,\tau)= f_\ld(y)\quad\mbox{ for }(y,\tau)\in \left(\R^n\setminus\{0\}\right)\times (-\infty,\infty).
$$ 
Note that if $u$ is the solution of \eqref{eq-fde-global-except-0} given by Theorem \ref{thm-existence-sol-fde-singular}, then \eqref{eq-existence-sols-trapped-self-similar} implies that 
\begin{equation}\label{eq-sol-trapped-particular-sols-rescaled-0-new}
f_{\ld_1}(y) \leq \tilde u(y,\tau )\leq f_{\ld_2}(y)\quad\forall (y,\tau)\in \left(\R^n\setminus\{0\}\right)\times(-\infty,\infty).
\end{equation} 
When $\frac{2}{1-m}<\gamma<n$, we will prove the large time behavior of the solution given by Theorem \ref{thm-existence-sol-fde-singular} with initial value satisfying \eqref{eq-initial-trapped-particular-sols-existence} for some constants $A_2>A_1>0$, 
in which case, the solution belongs to $C([0, \infty); L^1_{\loc}(\R^n))$ and is a weak solution of the Cauchy problem \eqref{eq-fde-global} (Corollary \ref{cor-existence-sol-fde-singular} in section \ref{sec-asymptotic-behavior}). More precisely we have the following main result.

\begin{thm}\label{thm-fde-rescaled-asymptotic}  
Let $n\geq3,$ $0< m<\frac{n-2}{n}$, $\frac{2}{1-m}<\gamma<n$, and   let $\al$, $\be$ be given by \eqref{eq-def-alpha-beta-gamma}. Let $u_0$ satisfy \eqref{eq-initial-trapped-particular-sols-existence}  and 
\begin{equation}\label{eq-sol-initial-value-similar-to-U-ld0-new}
u_0- A_{0} |x|^{-\gamma}\in L^1\left( r^{-\mu}; \R^n\right) 
\end{equation}
for some constants $A_2\ge A_0\ge A_1>0$ and $\mu_1<\mu<\mu_2$, where $\mu_1$, $\mu_2$ are given by \eqref{eq-weighted-L1-contraction-mu}.  Let $u$ be the solution of \eqref{eq-fde-global-except-0} which satisfies \eqref{eq-existence-sols-trapped-self-similar} with $\lambda_i=A_i^{1/ (\frac{2}{1-m}-\gamma )}$  for  $i=1,2$, and let $\tilde u(y,\tau)$ be given by \eqref{def-fde-rescaled}. Then as $\tau\to\infty$, $\tilde u(y,\tau)$ will converge  uniformly on  each compact subset of $\R^n\setminus\{0\}$ and in $L^1( r^{-\mu_1};\R^n)$ to $f_{\ld_0}(y)$ where  $\ld_0=A_0^{1/(\frac{2}{1-m}-\gamma)}$. 
\end{thm}

We end the introduction by stating some definitions and notations that will be used in the paper.
    
\begin{itemize}
\item For any $0\leq u_0\in L^1_{\loc}(\R^n\setminus\{0\}),$  we say that $u$ is a solution of \eqref{eq-fde-global-except-0}  
if $u>0$ in $(\R^n\setminus\{0\})\times(0, \infty)$ satisfies \eqref{eq-fde} in $ \left(\R^n\setminus\{0\}\right)\times(0, \infty)$ in the classical sense and 
\begin{equation}\label{eq-u-initial-value}
\| u(\cdot, t)-u_0\|_{L^1(K)}\to 0\quad\mbox{ as }t\to 0
\end{equation} 
for any compact set $K\subset \R^n\setminus\{0\}.$ 
 
\item
For any $0\leq u_0\in L^1_{\loc}(\R^n),$ we say that $u$ is a weak solution of \eqref{eq-fde-global}   if $ 0 \leq  u \in C\left([0, \infty); L^1_{\loc}(\R^n)\right)$ satisfies \eqref{eq-fde} in $  \R^n\times(0,\infty)$ in the distributional sense  
and \eqref{eq-u-initial-value} holds for any compact set $K\subset \R^n.$ 

\item
For any $x_0\in\R^n,$ and $R>0,$ we let $B_{R}(x_0)=\{x\in\R^n: |x-x_0|<R\}$ and $B_R=B_R(0).$ We also let $\cA_R= B_{R}\setminus \overline B_{1/R}$ for any $R>1$. 
\end{itemize}
The rest of the paper is organized as follows. In section \ref{sec-inversion}, we will study the inversion elliptic problem associated with  the solution of \eqref{eq-fde-scaled} which satisfies \eqref{eq-sol-behavior-origin} and  \eqref{eq-sol-behavior-infty} for some constants $A>0$ and $D_A>0$. Section \ref{sec-self-similar profiles} is devoted to the proof of Theorem \ref{thm-existence-self-similar}. In section \ref{sec-asymptotic-behavior} we will prove Theorem \ref{thm-weighted-L1-contraction-new}, Theorem \ref{thm-existence-sol-fde-singular} and Theorem \ref{thm-fde-rescaled-asymptotic}.    
 
\section{Inversion elliptic problem for self-similar profiles} \label{sec-inversion}
\setcounter{equation}{0}
\setcounter{thm}{0}

In order to study the existence of singular self-similar solutions
of \eqref{eq-fde}, we introduce an inversion formula for the solution of \eqref{eq-fde-scaled} which satisfies \eqref{eq-sol-behavior-origin} and \eqref{eq-sol-behavior-infty} for some constants $A>0$ and $D_A>0$. We first note that $f$ is a radially symmetric solution of \eqref{eq-fde-scaled} in $\R^n\setminus\{0\},$   
if and only if the function
\begin{equation}\label{eq-def-g-from-f}
g(r):=r^{-\frac{n-2}{m}}f(r^{-1}),\quad r=|x|>0,
\end{equation}   satisfies 
\begin{equation}\label{eq-inversion}
\La g^m+|x|^{\frac{n-2-nm}{m}-2}\left(\tilde\alpha g+\tilde\beta x\cdot \D g \right)=0,\quad g>0 
\end{equation}
in $\R^n\setminus\{0\}$ with
\begin{equation}\label{eq-tilde-alpha-beta}
\tilde\beta=-\beta,\quad\mbox{and}\quad\tilde\alpha=\alpha-\frac{n-2}{m}\,\beta.
\end{equation}
In this case the condition \eqref{eq-sol-behavior-origin} is equivalent to 
\begin{equation}\label{eq-fde-scaled-at-origin-inv}
\lim_{|x|\to\infty}|x|^{\frac{\tilde\alpha}{\tilde\beta}}g(x)=A.
\end{equation}  
Note that if \eqref{eq-alpha-beta-neg-condition} holds, then  
\begin{equation}\label{eq-tilde-alpha-beta-0}
\tilde\alpha>0,\quad\tilde\beta>0,\quad \frac{\tilde\alpha}{\tilde\beta}=-\frac{\alpha}{\beta}+\frac{n-2}{m}\in \left(0,\frac{n-2}{m}\right),
\end{equation}
and
\begin{equation}\label{eq-tilde-alpha-beta-1}
0<\frac{\tilde\alpha}{\tilde\beta}<\frac{n-2-nm}{m(1-m) }\quad \Leftrightarrow\quad  \frac{2}{1-m}<\frac{\alpha}{\beta}  <\frac{n-2}{m}\,.
\end{equation}

Hence existence of a radially symmetric  solution $f$ of \eqref{eq-fde-scaled}  in $\R^n\setminus\{0\}$  satisfying  \eqref{eq-sol-behavior-origin} and  \eqref{eq-sol-behavior-infty}  is equivalent  to the existence of a radially symmetric solution $g$ of \eqref{eq-inversion} in $\R^n\setminus\{0\}$ satisfying \eqref{eq-fde-scaled-at-origin-inv} and $g(0)=D_A$.   
In this section we will prove the  existence of a radially symmetric  solution $g$ to \eqref{eq-inversion} in $\R^n\setminus\{0\}$ satisfying \eqref{eq-fde-scaled-at-origin-inv} when \eqref{eq-alpha-beta-neg-condition} holds. 
 
\begin{lemma}\label{lem-v-derivative}
Let $n\geq3,$ $0<m\leq\frac{n-2}{n} $, $\tilde \alpha>0$, $\tilde\beta\not=0$ and $\tilde\alpha/\tilde\beta\leq \frac{n-2}{m}.$ 
For any $\eta>0$ and $R_0>0,$ let $g\in C([0,R_0);\R)\cap C^2((0,R_0);\R)$ be a solution to 
\begin{equation}\label{eq-fde-inversion}
(g^m)''+\frac{n-1}{r}(g^m)' +r^{\frac{n-2-nm}{m}-2}(\tilde\alpha g+\tilde\beta r g_{r})=0,\quad g>0
\end{equation}
in $(0,R_0)$ which satisfies 
\begin{equation}\label{eq-initial}
g(0)=\eta\quad\mbox{and }\quad  \lim_{r\to0+}r g_r(r)=0.
\end{equation} 
Let $\tilde k:=\tilde\beta/\tilde\alpha$. Then
\begin{equation*}
\left\{\begin{aligned}
g(r)+\tilde kr g'(r)>0\quad\forall r\in (0,R_0)\\
g'(r)<0\quad\quad\forall r\in (0,R_0).
\end{aligned}\right.
\end{equation*}
\end{lemma}
\begin{proof}
The proof is similar to  one for \cite[Lemma 2.1]{Hs2}. Let $h_1(r):=g(r)+\tilde k rg'(r).$
By direct computation $h_1$ satisfies 
\begin{equation}\label{eq-h_1}
h_1'+\left\{\frac{n-2-(m/\tilde k)}{r}-(1-m)\frac{g'}{g}+\frac{\tilde\beta}{m}r^{\frac{n-2-nm}{m}-1}g^{1-m}\right\}h_1=\left(n-2-\frac{m}{\tilde k}\right)\frac{g}{r}\geq0
\end{equation}
in $(0,R_0)$ since $\frac{1}{\tilde k}=\frac{\tilde\alpha}{\tilde\beta}\leq \frac{n-2}{m}$. By \eqref{eq-initial} there exists $\ve_0\in (0,R_0)$ such that  $h_1(\ve)>0$ for any $0<\ve\le\ve_0$. Let $0<\ve\le\ve_0$ and
$$
q(r):=g^{m-1}(r) \exp\left(\frac{\tilde \beta}{m} \int_\ve^r\rho^{\frac{n-2-nm}{m}-1} g^{1-m}(\rho)d\rho\right),\quad \forall r\in(\ve, R_0).
$$  
Multiplying \eqref{eq-h_1} by $r^{n-2-(m/\tilde k)}q(r)$, we have 
$$
\left(r^{n-2-(m/\tilde k)}q(r)h_1(r)\right)'\geq0\quad\mbox{in $(\ve,R_0)$},\quad\forall 0<\ve\le\ve_0
$$
which implies $h_1(r)>0$ for $\ve< r<R_0$ and $0<\ve\le\ve_0$. Hence $h_1(r)>0$ for any $0<r<R_0$. Since 
$$
\frac{1}{r^{n-1}}\left( r^{n-1}(g^m)'\right)'=-\tilde\alpha r^{\frac{n-2-nm}{m}-2} h_1<0\quad\forall r\in(0,R_0), 
$$
it follows from \eqref{eq-initial} that $r^{n-1}(g^m)'<0$ in $(0,R_0)$. Hence $g'<0$ in $(0,R_0)$ and the lemma follows. 
\end{proof}
 
In the following lemmas we will prove the local existence of solution of the O.D.E \eqref{eq-fde-inversion}.  

\begin{lemma}\label{lem-local-existence-inversion-m-small}
Let $n\geq3,$ $0<m<\frac{n-2}{n+1} $, and $\tilde \al,\tilde\be\in\R.$
For any $\eta>0,$ there exists a constant $\ve>0$ such that \eqref{eq-fde-inversion} has a unique solution $g\in C^1([0,\ve);\R)\cap C^2((0,\ve);\R)$ in $(0,\ve)$ which satisfies
\begin{equation}\label{eq-initial-m-small}
g(0)=\eta\quad \mbox{and}\quad g'(0)=0.
\end{equation} 
\end{lemma}
\begin{proof} 
Let $\tilde\eta:= {\eta}/{2},$ and let $\ve\in(0,1)$ be a constant to be chosen later. 
We first observe that if $g\in C^1([0,\ve);\R)\cap C^2((0,\ve);\R)$ is a solution of \eqref{eq-fde-inversion} in $(0,\ve)$ which satisfies \eqref{eq-initial-m-small}, then 
\begin{equation*}
(g^m(r))'=-\frac{1}{r^{n-1}}\int_0^r\rho^{n-3+\frac{n-2-nm}{m}}\left\{ \tilde\alpha g(\rho)+\tilde\beta \rho g_\rho(\rho)\right\}d\rho\quad\forall r\in(0,\ve)
\end{equation*} 
which suggests one to use fix point argument to prove existence of solution of \eqref{eq-fde-inversion}.  
We now define the Banach space 
$$
\cX_\ve:=\left\{(g,h): g, h\in C\left( [0,\ve]; \R\right) \right\}
$$ 
with a norm given by
$$||(g,h)||_{\cX_\ve}=\max\left\{\|g\|_{L^\infty([0, \ve])} ,\|h\|_{L^\infty\left([0, \ve]\right)} \right\}.$$  
For any $(g,h)\in \cX_\ve,$ we define  
$$\Phi(g,h):=\left(\Phi_1(g,h),\Phi_2(g,h)\right),$$ 
where for $0<r\leq\ve,$
\begin{equation}\label{eq-existence-contraction-map}
\left\{
\begin{aligned}
&\Phi_1(g,h)(r):=\eta+\int_0^r h(\rho)\,d\rho,\\
&\Phi_2(g,h)(r):=-\frac{g^{1-m}(r)}{mr^{n-1}}\int_0^r \rho^{ n-3+ \frac{n-2-nm}{m}}\left\{\tilde\alpha g(\rho)+\tilde\beta \rho h(\rho)\right\}\,d\rho.
\end{aligned}\right.
\end{equation}
Let $$\cD_{\ve,\eta}:=\left\{ (g,h)\in \cX_\ve:  ||(g,h)-(\eta,0)||_{\cX_{\ve}}\leq \tilde\eta=\eta/2 \right\}.$$
Note that $\cD_{\ve,\eta}$  is a closed subspace of $\cX_\ve$. We will show that if $\ve\in(0,1)$ is sufficiently small, the map $(g,h)\mapsto\Phi(g,h)$ 
will have  a unique  fixed point   in $\cD_{\ve,\eta}.$

We first  prove that $\Phi(\cD_{\ve,\eta})\subset \cD_{\ve,\eta}$ if $\ve\in(0,1)$ is sufficiently  small. In fact for any $\ve\in(0,1)$ and $(g,h)\in \cD_{\ve,\eta},$ 
\begin{align*}
&\max_{0\leq r\leq\ve} \left| \int_0^r h(\rho)\,d\rho\right| \leq \tilde\eta \ve  \leq \tilde\eta,
\end{align*}
and for ${0< r\leq\ve},$
\begin{align}\label{eq-existence-contraction-map2-est}
&\frac{g^{1-m}(r)}{mr^{n-1}}\int_0^r\rho^{n-3+\frac{n-2-nm}{m}}\left\{|\tilde\alpha|g(\rho)+|\tilde\beta\rho h(\rho)|\right\}d\rho\notag\\  
\leq&{(3\tilde\eta)^{1-m}}\left(\frac{3{}\tilde\eta|\tilde\alpha|\,r^{\frac{n-2-nm}{m}-1}}{n-2-2m }+\frac{ |\tilde\beta|\tilde\eta\,r^{\frac{n-2-nm}{m}}}{n-2-m}\right)\notag\\
\leq&(3\tilde\eta)^{2-m}\frac{|\tilde\alpha |+|\tilde\beta|}{n-2-2m} {r^{\frac{n-2-nm}{m}-1}}\notag\\
= &M r^{\frac{n-2-nm}{m}-1}\tilde\eta\leq  M \ve^{\frac{n-2-nm}{m}-1} \tilde\eta,
\end{align} 
where $M:= 3(3\tilde\eta)^{1-m}\frac{|\tilde\alpha| +|\tilde\beta| }{ n-2-2m} ,$ since $\tilde \eta\leq g\leq 3\tilde \eta$ for $(g,h)\in \cD_{\ve,\eta}.$
Since  $\frac{n-2-nm}{m}-1 >0$, by \eqref{eq-existence-contraction-map2-est} $\Phi(\cD_{\ve,\eta})\subset\cD_{\ve,\eta}$ for sufficiently small $\ve\in(0,1)$.   

Now we will prove that $\Phi\left|_{\cD_{\ve,\eta}}\right.$ is a contraction map  if $\ve\in (0,1)$ is sufficiently small.  Let $(g_1,h_1),(g_2,h_2)\in \cD_{\ve,\eta}$ and $\delta:=||(g_1,h_1)-(g_2,h_2)||_{\cX_\ve}$. Then   
\begin{equation*}
\|\Phi_1(g_1,h_1)- \Phi_1(g_2,h_2)\|_{L^\infty([0,\ve])}\leq  \max_{0\leq r\leq\ve }\int_0^r |h_1(\rho)-h_2(\rho)|d\rho\leq \ve\delta,
\end{equation*}
and  by \eqref{eq-existence-contraction-map2-est}, for $0< r\leq \ve$, 
\begin{align*}
&\left| \frac{g_1^{1-m}(r)}{mr^{n-1}}\int_0^r\rho^{n-3+\frac{n-2-nm}{m}}\left\{ \tilde\alpha g_1(\rho)+\tilde\beta \rho h_1(\rho)\right\}d\rho\right.\\
&\qquad\left.-\frac{g_2^{1-m}(r)}{mr^{n-1}}\int_0^r \rho^{n-3+\frac{n-2-nm}{m}}\left\{\tilde\alpha g_2(\rho)+\tilde\beta \rho h_2(\rho)\right\}d\rho  
  \right|\\
\leq&\frac{|g_1^{1-m}(r)-g_2^{1-m}(r)|}{mr^{n-1}}\int_0^r \rho^{ n-3+ \frac{n-2-nm}{m}}\left\{|\tilde\alpha| g_1
+|\tilde\beta \rho h_1|\right\}d\rho \\
&\qquad+\delta\cdot  \frac{g_2^{1-m}(r)}{mr^{n-1}}\int_0^r \rho^{ n-3+ \frac{n-2-nm}{m}}\left( |\tilde\alpha|+|\tilde\beta| \rho  \right)d\rho\\
\leq& \frac{(1-m)}{\tilde\eta^m}  \frac{|g_1(r)-g_2(r)|}{  g_1^{1-m}(r)} M\ve^{\frac{n-2-nm}{m}-1}\tilde\eta 
+\frac{\delta\,(3\tilde\eta)^{1-m}}{mr^{n-1}}\int_0^r\rho^{n-3+\frac{n-2-nm}{m}}\left(| \tilde\alpha|+|\tilde\beta| \rho  \right)d\rho    \\
\leq&\left\{(1-m)M+\frac{(3\tilde\eta)^{1-m}}{n-2-2m} \left( |\tilde\alpha|+|\tilde\beta|\ve\right)
\right\} \ve^{\frac{n-2-nm}{m}-1}  \,\delta
\end{align*}
since $\tilde\eta\leq g_1(r), g_2(r)\leq 3\tilde\eta$ for any $r\in[0,\ve]$. Hence 
\begin{align*}
&\|\Phi_2(g_1,h_1)- \Phi_2(g_2,h_2)\|_{L^\infty([0,\ve])}\\
\leq&\max_{0 < r\leq\ve } 
\left|\frac{g_1^{1-m}(r)}{mr^{n-1}}\int_0^r \rho^{n-3+\frac{n-2-nm}{m}}\left\{ \tilde\alpha g_1(\rho)+\tilde\beta \rho h_1(\rho)\right\}d\rho\right.\\
&\qquad\left.-\frac{g_2^{1-m}(r)}{mr^{n-1}}\int_0^r\rho^{n-3+\frac{n-2-nm}{m}}\left\{ \tilde\alpha g_2(\rho)+\tilde\beta\rho h_2(\rho)\right\}d\rho\right|\\
\leq&\left\{(1-m)M+\frac{(3\tilde\eta)^{1-m}}{n-2-2m} \left( |\tilde\alpha|+|\tilde\beta|\ve\right)
\right\} \ve^{\frac{n-2-nm}{m}-1}  \,\delta.
\end{align*}
Since $\frac{n-2-nm}{m}-1>0$, by choosing $0<\ve<1$ sufficiently small, we obtain that $\Phi$ is Lipschitz continuous on $\cD_{\ve,\eta}$ with a Lipschitz constant which is less than $1/2$. Hence by the contraction map theorem there exists a unique fixed point $(g,h)=\Phi(g,h)$ in $ \cD_{\ve,\eta}.$  Then
\begin{align}
&\left\{\begin{aligned}
g(r)=&\eta+\int_0^rh(\rho)\,d\rho\quad\forall 0\le r<\ve\\
h(r)=&-\frac{ g^{1-m}(r)}{mr^{n-1}}\int_0^r\rho^{n-3+\frac{n-2-nm}{m}}(\tilde\alpha g(\rho)+\tilde\beta\rho h(\rho))\,d\rho \quad\forall 0<r<\ve
\end{aligned}\right.\label{eq-g-h-representation}\\
\Rightarrow\quad&g'(r)=h(r)=-\frac{g^{1-m}(r)}{mr^{n-1}}\int_0^r\rho^{n-3+\frac{n-2-nm}{m}}(\tilde\alpha g(\rho)+\tilde\beta\rho g'(\rho))\,d\rho \quad\forall 0<r<\ve\label{eqn-g-integral-representation}\\
\Rightarrow\quad&r^{n-1}(g^m)'(r)=-\int_0^r\rho^{n-3+\frac{n-2-nm}{m}}(\tilde\alpha g(\rho)+\tilde\beta\rho g'(\rho))\,d\rho \quad\forall 0<r<\ve.
\label{eqn-g-integral-representation2}
\end{align}
By \eqref{eq-g-h-representation} and \eqref{eqn-g-integral-representation}, $g(0)=\eta$ and $g'(r)$ is   continuously differentiable in $(0,\ve)$. Since 
$g'=h\in C([0,\ve);\R)$ in $(0,\ve)$, by \eqref{eqn-g-integral-representation},
\begin{equation*}
|g'(r)|\le\frac{C}{r^{n-1}}\int_0^r\rho^{n-3+\frac{n-2-nm}{m}}\,d\rho\le C'r^{\frac{n-2-nm}{m}-1}\to 0\quad\mbox{ as }r\to 0
\end{equation*}
and then $g$ belongs to $C^1([0,\ve);\R)\cap C^2((0,\ve);\R)$ and satisfies  \eqref{eq-initial-m-small}. 
Differentiating \eqref{eqn-g-integral-representation2} with respect to 
$r\in (0,\ve)$, we get that $g$ satisfies \eqref{eq-fde-inversion}  in $(0,\ve)$. Hence $g\in C^1([0,\ve);\R)\cap C^2((0,\ve);\R)$ is the unique solution of  \eqref{eq-fde-inversion}  in $(0,\ve)$  which satisfies  \eqref{eq-initial-m-small}.
\end{proof} 

\begin{lemma}\label{lem-local-existence-inversion-m-large}
Let $n\geq3,$  $\frac{n-2}{n+1}\leq m< \frac{n-2}{n}$, and $\tilde\al,\tilde\be\in\R.$ For any $\eta> 0,$  there exists a constant $\ve>0$ such that
\eqref{eq-fde-inversion} has a unique solution $g\in C^{0, \delta_0}([0,\ve);\R)\cap C^2((0,\ve);\R)$ in $(0,\ve)$ which satisfies
\begin{equation}\label{eq-initial-m-large}
g(0)=\eta\quad \mbox{ and } \quad \lim_{r\to 0^+}r^{\delta_1} g_r(r)=-\frac{\tilde\alpha \eta^{2-m}}{n-2-2m}
\end{equation} 
where 
\begin{equation}\label{defn-delta0-1}
\delta_1=1-\frac{n-2-nm}{m}\in[0,1)\quad\mbox{ and }\quad\delta_0=\frac{1-\delta_1}{2}=\frac{n-2-nm}{2m}\in(0,1/2].
\end{equation} 
\end{lemma}
\begin{proof}  
Let $\tilde\eta:= {\eta}/{2}$ and let $ \ve\in(0,1)$ be a constant to be chosen later. We define the Banach space 
\begin{align*}
\cX'_\ve&:=\left\{(g,h): g\in C^{0, \delta_0}\left( [0,\ve]; \R \right), h\in C\left(  
(0,\ve]; \R \right), \mbox{ and}\right.\\
&\quad\left. \,\,\,\mbox{$r^{  \delta_1 } h(r)$  can be extended to a function in $C\left( [0,\ve]; \R \right) $}\right\},
\end{align*}
with norm 
$$
||(g,h)||_{\cX'_\ve}=\max\left\{ \| g\|_{C^{0, \delta_0}([0,\ve])},\|r^{  \delta_1} h\|_{L^\infty([0,\ve])}\right\},
$$
where 
$$
\| g\|_{C^{0, \delta_0}([0,\ve])}= \|g\|_{L^\infty ([0,\ve])}+[g]_{\delta_0, [0,\ve]}=\| g\|_{L^\infty ([0,\ve])}+\sup_{r,s\in[0,\ve],\, r\not=s} \frac{|g(r)-g(s)|}{|r-s|^{\delta_0}},
$$ 
and we will still denote the  extension  of  $r^{\delta_1}h$ by $r^{\delta_1}h.$  For any $(g,h)\in \cX'_\ve,$ let $\Phi(g,h):=\left(\Phi_1(g,h),\Phi_2(g,h)\right)$  be given by \eqref{eq-existence-contraction-map}, which is well-defined  for $0<r\leq\ve$ since $0\leq \delta_1<1$. Let 
$$
\cD'_{\ve,\eta}:=\left\{ (g,h)\in \cX'_\ve:  ||(g,h)-(\eta,-\zeta r^{-\delta_1})||_{\cX'_{\ve}}\leq \tilde\eta=\eta/2  ,\quad g(0)=\eta\right\},
$$
where
$$\zeta:= \frac{\tilde\alpha \eta^{2-m}}{n-2-2m}.$$ 
We will show that  for   $\eta>0,$ there exists    $\ve\in(0,1)$ such that the map  $(g,h)\mapsto\Phi(g,h)$ has a unique fixed point  in the closed subspace $\cD'_{\ve,\eta}.$
   
We first prove  that  $\Phi(\cD'_{\ve,\eta})\subset \cD'_{\ve,\eta}$ for sufficiently small $\ve\in (0,1)$. For any $(g,h)\in \cD'_{\ve,\eta},$  $0\leq r<\ve ,$ and  $0<s\leq\ve-r,$
\begin{equation*}
\begin{aligned}
s^{-\delta_0}\left|\Phi_1(g,h)(r+s)-\Phi_1(g,h)(r)\right|
=&s^{-\delta_0} \left| \int_r^{r+s}h(\rho)d\rho\right| \\
\leq &s^{-\frac{1-\delta_1}{2}}\left(\tilde\eta + | \zeta|\right)  \int_r^{r+s} \rho^{-\delta_1} d\rho \\
 =&\frac{ \tilde\eta +  |\zeta|  }{1-\delta_1} s^{\frac{1-\delta_1}{2}}\left\{ \left(\frac{r}{s}+1\right)^{1-\delta_1}-\left(\frac{r}{s}\right)^{1-\delta_1} \right\}\\
 \leq& \frac{ \tilde\eta + | \zeta| }{1-\delta_1} s^{\frac{1-\delta_1}{2}} 
\end{aligned}
\end{equation*}
since $(1+z)^{1-\delta_1}\leq 1+z^{1-\delta_1} $ for any $z>0$. Hence
\begin{equation*}
\left\{\begin{aligned}
 &\|\Phi_1(g,h)-\eta\|_{L^\infty([0,\ve])}\leq   \frac{ \tilde\eta + | \zeta|  }{1-\delta_1} \ve ^{ 1-\delta_1}=  \frac{m( \tilde\eta +  |\zeta|)  }{n-2-nm} \ve ^{ \frac{n-2-nm}{m}}\leq\frac{\tilde\eta}{2}\\
& \left[\Phi_1(g,h)-\eta\right]_{\delta_0 ,  [0,\ve]}= \left[\Phi_1(g,h)\right]_{\delta_0 ,  [0,\ve]}
\leq  \frac{ \tilde\eta + | \zeta|  }{1-\delta_1} \ve ^{\frac{1-\delta_1}{2}} =  \frac{m( \tilde\eta + | \zeta|)  }{n-2-nm} \ve ^{ \frac{n-2-nm}{2m}}\leq \frac{\tilde\eta}{2}
\end{aligned}\right.
\end{equation*} 
if $0<\ve^{\frac {n-2-nm}{2m}}<\frac{ (n-2-nm)\,\eta}{4m\,(\eta+|\zeta|)}. $ Thus 
\begin{equation}\label{eq-ex-g-m-large-unif-conv-0}
\left\|\Phi_1(g,h)-\eta\right\|_{C^{0,\delta_0 }(  [0,\ve])}\leq \tilde\eta
\end{equation}
if $0<\ve^{\frac {n-2-nm}{2m}}<\frac{ (n-2-nm)\,\eta}{4m\,(\eta+|\zeta|)}$. 
By  the l'Hospital rule,    for any $(g,h)\in\cD'_{\ve,\eta} $, 
\begin{align}
&\lim_{r\to 0^+} \frac{g^{1-m}(r)}{mr^{n-1-\delta_1}}\int_0^r \rho^{ n-3+ \frac{n-2-nm}{m}}\left\{ \tilde\alpha g(\rho)+\tilde\beta \rho h(\rho)\right\}d\rho\notag\\
=& \frac{\eta^{1-m}}{m}  \lim_{r\to0^+} \frac{ r^{ n-2-\delta_1}\left\{ \tilde\alpha g(r)+\tilde\beta r h(r)\right\}}{(n-1-\delta_1)r^{n-2-\delta_1}}\notag\\
=&  \frac{\eta^{1-m}}{m({n-1-\delta_1} )}\left\{\tilde\alpha\eta+ \tilde\beta\lim_{r\to0^+} rh(r)\right\}\notag\\
=&\frac{\tilde \alpha \eta^{2-m}}{n-2-2m}=\zeta\label{eq-existence-Phi-contraction-m-large-unif.-conv} 
\end{align}
since $0  \leq \delta_1  <1$ and $\displaystyle\lim_{r\to0}|rh(r)|\leq \lim_{r\to0} \left(\tilde\eta  +|\zeta|\right) r^{1-\delta_1} =0.$  This implies that   for any $(g,h)\in\cD'_{\ve,\eta}$, $r^{\delta_1}\Phi_2(g,h)\in C([0,\ve];\R)$ with $\lim_{r\to0} r^{\delta_1}\Phi_2(g,h)(r)=-\zeta$.   
Now we claim that the convergence in \eqref{eq-existence-Phi-contraction-m-large-unif.-conv} is uniform for any $(g,h)\in\cD'_{\ve,\eta} $.  We first observe that for any $0<r\leq\ve,$
\begin{align}
\frac{g^{1-m}(r)}{mr^{n-1-\delta_1}}\int_0^r\rho^{n-2+\frac{n-2-nm}{m}-\delta_1}\rho^{ \delta_1}|h(\rho)|d\rho
\leq&\frac{(3\tilde\eta)^{1-m}(\tilde\eta+|\zeta|) r^{\frac{n-2-nm}{m}}}{m\left(n-1+\frac{n-2-nm}{m}-\delta_1\right)}\notag\\
=&\frac{\left({3\tilde\eta}\right)^{1-m}\left(\tilde\eta+|\zeta|\right)}{m \left(n-2\delta_1\right)}{r^{\frac{n-2-nm}{m}}}\label{eq-ex-g-m-large-unif-conv-1}
\end{align}
and 
\begin{align}
&\left|   \frac{\tilde\alpha g^{1-m}(r)}{mr^{n-1-\delta_1}}\int_0^r \rho^{ n-3+ \frac{n-2-nm}{m}} g(\rho)\,d\rho -\zeta\right|\notag\\
=&|\tilde\alpha|\cdot \left|     \frac{g^{1-m}(r)}{mr^{n-1-\delta_1}}\int_0^r \rho^{ n-3+ \frac{n-2-nm}{m}} g(\rho)\,d\rho -\frac{   g^{2-m}(0)}{n-2-2m}\right|\notag\\ 
\leq &|\tilde\alpha|     \cdot  \frac{|g^{1-m}(r)-g^{1-m}(0)|}{mr^{n-1-\delta_1}}\int_0^r \rho^{ n-3+ \frac{n-2-nm}{m}} g(\rho)\,d\rho\notag \\  
&\qquad +|\tilde\alpha|\cdot       \frac{g^{1-m}(0)}{mr^{n-1-\delta_1}}\int_0^r \rho^{ n-3+ \frac{n-2-nm}{m}}| g(\rho)-g(0)|\,d\rho\label{eq-ex-g-m-large-unif-conv-1-1}  
\end{align}
since  $g(0)=\eta,$ $ \tilde\eta\leq g\leq 3\tilde\eta$  in  $[0,\ve],$ and 
$$\frac{1}{mr^{n- 1-\delta_1}}\int_0^r \rho^{ n-3+ \frac{n-2-nm}{m}} d\rho=\frac{1}{mr^{n- 2+\frac{n-2-nm}{m}}}\int_0^r \rho^{ n-3+ \frac{n-2-nm}{m}} d\rho=\frac{1}{n-2-2m} .$$
By  \eqref{eq-ex-g-m-large-unif-conv-1-1} and the mean value theorem, 
\begin{align}
&\left|  \frac{\tilde\alpha g^{1-m}(r)}{mr^{n-1-\delta_1}}\int_0^r \rho^{ n-3+ \frac{n-2-nm}{m}} g(\rho)\, d\rho -\zeta\right|\notag\\
\le&| \tilde\alpha|        {(1-m) \tilde\eta^{-m} |g(r)-g(0)|} \cdot\frac{3 \tilde\eta}{n-2-2m}\notag\\  
&\qquad +|\tilde\alpha |      \frac{g^{1-m}(0)}{mr^{n-1-\delta_1}}\int_0^r \rho^{ n-3+ \frac{n-2-nm}{m} }|g(\rho)-g(0)|\, d\rho
\label{eq-ex-g-m-large-unif-conv-1-2}  
\end{align}
since $\tilde\eta\leq g(r)\leq 3 \tilde\eta$ for $r\in[0,\ve]$. 
Since $[g]_{  \delta_0, [0,\ve]}
\leq \tilde\eta$ for $(g,h)\in\cD'_{\ve,\eta},$ 
    the right hand side of \eqref{eq-ex-g-m-large-unif-conv-1-2} is bounded above by
\begin{align}
\leq&\frac{| \tilde\alpha |(1-m){3\tilde \eta^{2-m}} }{n-2-2m} r^{\delta_0} +\frac{|\tilde\alpha|\eta^{1-m}\tilde\eta}{n-2-2m}r^{\delta_0}\le\frac{2|\tilde\alpha|\eta^{2-m } }{n-2-2m}r^{\frac{n-2-nm}{2m}}\quad\forall 0<r\le\ve.
\label{eq-ex-g-m-large-unif-conv-2}
\end{align}
By \eqref{eq-ex-g-m-large-unif-conv-1}, \eqref{eq-ex-g-m-large-unif-conv-1-2}  and \eqref{eq-ex-g-m-large-unif-conv-2},  we deduce    uniform  convergence in  \eqref{eq-existence-Phi-contraction-m-large-unif.-conv}   for any $(g,h)\in\cD'_{\ve,\eta} $.  
By  \eqref{eq-ex-g-m-large-unif-conv-1}, \eqref{eq-ex-g-m-large-unif-conv-1-2} and \eqref{eq-ex-g-m-large-unif-conv-2}, for any  $\eta>0$, $(g,h)\in\cD'_{\ve,\eta} $, $ {0< r\leq\ve} , $ 
 \begin{align}
 r^{  \delta_1}\left|\Phi_2(g,h)(r)+\zeta r^{-\delta_1}\right|&=  \left|  \frac{g^{1-m}(r)}{mr^{n-1-\delta_1}}\int_0^r \rho^{ n-3+ \frac{n-2-nm}{m}}\left\{ \tilde\alpha g(\rho)+\tilde\beta \rho h(\rho)\right\}d\rho -\zeta\right|\notag\\
 & \leq  \left\{ \frac{2|\tilde\alpha| \eta^{2-m } }{n-2-2m} + \frac{|\tilde\beta|( \eta/2+|\zeta| )}{m \left(n-2\delta_1 \right)}\left( \frac{3 \eta}{2}\right)^{1-m} \right\} {\ve^{\frac{n-2-nm}{2m}}}\leq \tilde \eta,\label{eq-2nd-variable-bd}
 \end{align} 
if $\ve\in(0,1)$ is sufficiently small. Hence by \eqref{eq-ex-g-m-large-unif-conv-0} and \eqref{eq-2nd-variable-bd}, for any $\eta>0$, $\Phi(\cD'_{\ve,\eta})\subset\cD'_{\ve,\eta}$  if $\ve\in(0,1)$ is sufficiently small.

Now we will show that $\Phi\left|_{\cD'_{\ve,\eta}}\right.$ is a contraction map  if  $\ve\in (0,1)$ is sufficiently small. Let $(g_1,h_1),(g_2,h_2)\in \cD'_{\ve,\eta}$ and $\delta:=||(g_1,h_1)-(g_2,h_2)||_{\cX'_\ve}$. Then
\begin{align*}
 &\|\Phi_1(g_1,h_1)- \Phi_1(g_2,h_2)\|_{L^\infty([0,\ve])} +   \left[\Phi_1(g_1,h_1)-\Phi_1(g_2,h_2)\right]_{\delta_0 ,  [0,\ve]}\\
 &=  \max_{0\leq r\leq\ve }\left|\int_0^r \left\{h_1(\rho)-h_2(\rho)\right\}d\rho\right|+\sup_{0\leq r<\ve,\,0<s\leq\ve-r}s^{-\delta_0} \left| \int_r^{r+s}\left\{h_1(\rho)-h_2(\rho)\right\}d\rho\right| \\
 &\leq  \left( \int_0^r \rho^{-  \delta_1}d\rho+ \sup_{0\leq r<\ve,\,0<s\leq\ve-r} s^{-\frac{1-\delta_1}{2}}   \int_r^{r+s} \rho^{-\delta_1} d\rho \right) \delta\\
 &=   \left[\frac{r^{1- \delta_1}}{1- \delta_1}+ \sup_{0\leq r<\ve,\,0<s\leq\ve-r} \frac{ s^{\frac{1-\delta_1}{2}} }{1-\delta_1} \left\{ \left(\frac{r}{s}+1\right)^{1-\delta_1}-\left(\frac{r}{s}\right)^{1-\delta_1} \right\}\right] \delta \\
 &\leq   \left(\frac{\ve^{1- \delta_1}}{1- \delta_1}+   \frac{\ve^{\frac{1-\delta_1}{2}} }{1-\delta_1} \right) \delta\leq \frac{2\ve^{\frac{n-2-nm}{2m}} }{1-\delta_1}  \,\delta   \leq\frac{\delta}{2}
 \end{align*} 
  if $\ve>0$ is sufficiently small,
  and 
for any $0<r\leq \ve, $ 
   \begin{align*}
& r^{ \delta_1}|\Phi_2(g_1,h_1)- \Phi_2(g_2,h_2)|(r)\\
=&\left| \frac{g_1^{1-m}(r)}{mr^{n-1- \delta_1}}\int_0^r \rho^{ n-3+ \frac{n-2-nm}{m}}\left\{ \tilde\alpha g_1(\rho)+\tilde\beta \rho h_1(\rho)\right\}d\rho\right.\\
&\qquad \left.-\frac{g_2^{1-m}(r)}{mr^{n-1-\delta_1}}\int_0^r\rho^{n-3+\frac{n-2-nm}{m}}\left\{ \tilde\alpha g_2(\rho)+\tilde\beta\rho h_2(\rho)\right\}d\rho\right|\\
\leq&\frac{|g_1^{1-m}(r)-g_2^{1-m}(r)|}{mr^{n-1- \delta_1}}\int_0^r \rho^{ n-3+ \frac{n-2-nm}{m}}\left\{| \tilde\alpha| g_1(\rho)+|\tilde\beta \rho h_1(\rho)|\right\}d\rho  \\
&\qquad +\frac{  g_2^{1-m}(r)}{mr^{n-1-\delta_1}}\int_0^r\rho^{n-3+\frac{n-2-nm}{m}}\left\{| \tilde\alpha||g_1(\rho)-g_2(\rho)|+|\tilde\beta |\rho|h_1(\rho)-h_2(\rho)| \right\}d\rho \\
=:&I_1+I_2. 
\end{align*}
Since $\eta/2 \leq g_1, g_2\leq 3\eta/2$ in $[0,\ve]$, by  the mean value theorem  
\begin{equation}\label{eq-ex-g-m-large-I_1}
I_1 \leq\frac{(1-m)|g_1(r)-g_2(r)|}{  m(\eta/2)^mr^{n-1- \delta_1}}  \int_0^r \rho^{ n-3+ \frac{n-2-nm}{m}}\left\{ |\tilde\alpha |g_1(\rho)+|\tilde\beta \rho h_1(\rho)|\right\}d\rho. 
      \end{equation}
Since $g_1(0)=g_2(0)=\eta,$ 
\begin{equation}\label{eq-g1-g2-l-infty-bd}
\|g_1-g_2\|_{L^\infty([0,\ve])}\leq\sup_{0\leq r\leq\ve}r^{\delta_0 } [g_1-g_2]_{\delta_0,[0,r]}\leq \ve^{\delta_0}\delta. 
\end{equation}
Hence  it follows from  \eqref{eq-ex-g-m-large-unif-conv-1}, \eqref{eq-ex-g-m-large-unif-conv-1-2}, \eqref{eq-ex-g-m-large-unif-conv-2}, \eqref{eq-ex-g-m-large-I_1} and \eqref{eq-g1-g2-l-infty-bd}, that for $0<r\leq\ve,$
    \begin{equation}\label{eq-ex-g-m-large-I_1-contraction}
    \begin{aligned}
I_1\leq&\frac{(1-m)|g_1(r)-g_2(r)|}{ (\eta/2)^mg_1^{1-m}(r)} \cdot \frac{g_1^{1-m}(r)}{  mr^{n-1-\delta_1}}  \int_0^r \rho^{ n-3+ \frac{n-2-nm}{m}}\left( |\tilde\alpha| g_1(\rho)+|\tilde\beta \rho h_1(\rho)|\right)d\rho      \\  
 \leq &   \frac{(1-m)|g_1(r)-g_2(r)|}{ (\eta/2)^mg_1^{1-m}(r)} \cdot\left[| \zeta| +\left\{ \frac{2|\tilde\alpha| \eta^{2-m } }{n-2-2m} + \frac{|\tilde\beta|( \eta+|\zeta|)}{m (n-2\delta_1) }\left( \frac{3 \eta}{2}\right)^{1-m} \right\} {r^{\frac{n-2-nm}{2m}}}\right]\\
  \leq &\frac{2(1-m)}{\eta}\left[ \frac{|\tilde\alpha |\eta^{2-m}}{n-2-2m} +\left\{ \frac{2|\tilde\alpha| \eta^{2-m } }{n-2-2m} + \frac{|\tilde\beta|( \eta+|\zeta|)}{m  (n-2\delta_1) }\left( \frac{3 \eta}{2}\right)^{1-m} \right\} {\ve^{\frac{n-2-nm}{2m}}}\right] \\
  &\qquad\cdot  \|g_1-g_2\|_{L^\infty([0,\ve])}\\
    \leq&  \frac{2(1-m)}{\eta}\left[ \frac{|\tilde\alpha| \eta^{2-m}}{n-2-2m} +\left\{ \frac{2|\tilde\alpha| \eta^{2-m } }{n-2-2m} + \frac{|\tilde\beta|( \eta+|\zeta|)}{m  (n-2\delta_1   ) }\left( \frac{3 \eta}{2}\right)^{1-m} \right\} {\ve^{\frac{n-2-nm}{2m}}}\right] \ve^{\delta_0}\delta
\end{aligned}
\end{equation}
and 
\begin{align}\label{eq-ex-g-m-large-I_2-contraction}
I_2 &\leq \delta\cdot \frac{  g_2^{1-m}(r)}{mr^{n-1- \delta_1}}\int_0^r \rho^{ n-3+ \frac{n-2-nm}{m}}\left\{ |\tilde\alpha|\rho^{\delta_0}+|\tilde\beta |\rho^{1-\delta_1} \right\}d\rho\notag\\
&\leq \left(\frac{3\eta}{2}\right)^{1-m}\cdot\frac{ |\tilde\alpha|+|\tilde\beta|}{n-2-2m} \ve^{\frac{n-2-nm}{2m}}  \delta. 
\end{align}
By \eqref{eq-ex-g-m-large-I_1-contraction} and  \eqref{eq-ex-g-m-large-I_2-contraction}, for any $\eta>0,$  there exists  sufficiently small  $\ve\in (0,1)$ such that  for $0< r\leq \ve$,  
\begin{equation*}
I_1\leq \frac{1}{4} \delta\quad\mbox{ and }\quad I_2\leq \frac{1}{4} \delta. 
\end{equation*}
Thus by choosing  sufficiently small $\ve\in(0,1),$  the map $\Phi$ is Lipschitz continuous on $\cD'_{\ve,\eta}$ with a Lipschitz constant which is less than $1/2$. Hence  by  the contraction map theorem there exists a unique fixed point $(g,h)=\Phi(g,h)$ in $ \cD'_{\ve,\eta}$. Then by an argument similar to the proof of Lemma \ref{lem-local-existence-inversion-m-small}, $g $ belongs to $C^{0, \delta_0}([0,\ve);\R)\cap C^2((0,\ve);\R)$ and
  satisfies  \eqref{eq-fde-inversion}  in $(0,\ve)$. By \eqref{eq-existence-Phi-contraction-m-large-unif.-conv}, 
\eqref{eq-initial-m-large} holds. 
      
Finally we observe that if  $\tilde g\in C^{0, \delta_0}([0,\ve);\R)\cap C^2((0,\ve);\R)$ is  a  solution of  \eqref{eq-fde-inversion}  in $(0,\ve)$ which satisfies \eqref{eq-initial-m-large},  then   $(\tilde g, \tilde g') \in \cD'_{\ve,\eta}$ for sufficiently small $\ve>0$.     
Then uniqueness of a solution of  \eqref{eq-fde-inversion}  in $(0,\ve)$ satisfying   \eqref{eq-initial-m-large} follows from the contraction map theorem. 
\end{proof} 

Now we are ready to  prove the global existence of radially symmetric singular  solution of \eqref{eq-inversion} in $\R^n\setminus\{0\}$ for $\tilde\al>0$ and $\tilde\be>0$. 

\begin{thm} \label{thm-existence-inversion}
Let $n\geq3,$ $0< m <\frac{n-2}{n} ,$    $\tilde \alpha>0,  \tilde\beta>0 $, $\frac{\tilde\alpha}{\tilde\beta}\leq \frac{n-2}{m}$ and $\eta>0$. 
\begin{enumerate}[(a)] 
\item If $0<m<\frac{n-2}{n+1}$, then there exists a unique solution $g\in C^1([0,\infty);\R)\cap C^2((0,\infty);\R)$  of \eqref{eq-fde-inversion}  in $(0,\infty)$ which satisfies \eqref{eq-initial-m-small}. 
\item If $\frac{n-2}{n+1}\leq m< \frac{n-2}{n}$, then there exists a unique solution $g\in C^{0, \delta_0}([0,\infty);\R)\cap C^2((0,\infty);\R)$ of \eqref{eq-fde-inversion}  in $(0,\infty)$ which satisfies \eqref{eq-initial-m-large} where $\delta_0$ and $\delta_1$ are given by \eqref{defn-delta0-1}.
\end{enumerate}
Moreover the function 
$$
w_1(r):=r^2g^{2\tilde k}(r)
$$   
satisfies $w_1'(r)>0$ for $r>0,$ where $\tilde k:=\tilde\beta/\tilde\alpha$. 
\end{thm}
\begin{proof}  
We first consider the case $0<m<\frac{n-2}{n+1}.$ By Lemma \ref{lem-local-existence-inversion-m-small} there exits a unique  solution $g\in C^1([0,\ve);\R)\cap C^2((0,\ve);\R)$ of  \eqref{eq-fde-inversion} in $(0,\ve)$ for some $\ve>0$, which satisfies   \eqref{eq-initial-m-small}.  Let $(0,R_0)$ be the maximal interval of   existence of solution $g\in C^1([0,R_0);\R)\cap C^2((0,R_0);\R)$ of  \eqref{eq-fde-inversion}  satisfying  \eqref{eq-initial-m-small}.  We claim that $R_0=\infty. $ Suppose to the contrary  that $R_0<\infty. $ Then there is a sequence $\{r_i\}_{i=1}^\infty$  such that $r_i<R_0$, $r_i\nearrow R_0$ as $i\to\infty$, and either   
\begin{equation*}
|g'(r_i)|\to\infty\quad\mbox{ as }i\to\infty
\end{equation*} 
or
\begin{equation*}
g(r_i)\to 0\quad\mbox{ as }i\to\infty
\end{equation*}
or
\begin{equation*}
g(r_i)\to\infty\quad\mbox{ as }i\to\infty
\end{equation*}
holds. By Lemma \ref{lem-v-derivative},
\begin{align*}
&g'(r)<0 \quad\mbox{and }\quad w_1'(r)=2r g^{2\tilde k-1}(r)(g(r)+\tilde kr g'(r))>0\quad\forall r\in(0,R_0)\\
\Rightarrow\quad&r^2g^{2\tilde k}(r)= w_1(r)> w_1(R_0/2)>0\qquad\qquad\qquad\forall r\in(R_0/2,R_0).
\end{align*}
Hence
$$  0<\left\{ R_0^{-2} w_1(R_0/2) \right\} ^{1/(2\tilde k)} <   g(r) <g(0)=\eta\quad\forall r\in (R_0/2, R_0).$$  
Using   Lemma  \ref{lem-v-derivative} again, we have 
$$ 
-\frac{2g(0)}{R_0}< -\frac{g(r)}{r}< \tilde k g'(r)<0\quad\forall r\in (R_0/2, R_0).
$$  
Thus contradiction arises. Hence we conclude that $R_0=\infty$. Hence there exists a unique solution $g\in C^1([0,\infty);\R)\cap C^2((0,\infty);\R)$ of  \eqref{eq-fde-inversion}   in $(0,\infty)$ satisfying   \eqref{eq-initial-m-small}. From  Lemma \ref{lem-v-derivative}, it follows  that $w'_1(r)>0$ for any $r>0.$ 

When $\frac{n-2}{n+1}\leq m<\frac{n-2}{n},$  since \eqref{eq-initial-m-large} implies  \eqref{eq-initial},  a  similar argument  as above using Lemma \ref{lem-v-derivative} and Lemma   \ref{lem-local-existence-inversion-m-large}    implies   the existence and uniqueness of  a   global  solution $g\in C^{0, \delta_0}([0,\infty);\R)\cap C^2((0,\infty);\R)$ of  \eqref{eq-fde-inversion} in  $(0,\infty)$, which satisfies   \eqref{eq-initial-m-large} and $w'_1(r)>0$ for any $r>0.$
 \end{proof}

Under the assumption that  $0<\frac{\tilde\alpha}{\tilde\beta}<\frac{n-2-nm}{m\,(1-m)},$ we will now prove  the decay rate  of the  solution of   \eqref{eq-fde-inversion} in $(0,\infty)$ which satisfies  \eqref{eq-initial}  as $r\to\infty$.

\begin{prop}\label{prop-behavior-infinity-inversion}
Let $n\geq3,$ $0< m <\frac{n-2}{n} ,$    $\tilde \alpha>0,\tilde\beta>0 $ and  $\frac{\tilde\alpha}{\tilde\beta}<\frac{n-2-nm}{m\,(1-m)}.$ 
Let  $g\in C([0,\infty);\R)\cap C^2((0,\infty);\R)$ be a solution of  \eqref{eq-fde-inversion}  in $(0,\infty)$  satisfying  \eqref{eq-initial}. Then there exists a constant $A>0$ such that \eqref{eq-fde-scaled-at-origin-inv} holds.
\end{prop}
\begin{proof}
We will use a modification of the proof of  Theorem 1.6 of \cite{Hs2} to prove this proposition. Let 
\begin{equation}\label{eq-def-tilde-q}
\tilde q(r):=r^{ \frac{\tilde\alpha}{\tilde \beta}}g(r),\quad\mbox{ and} \quad\tilde k=\frac{\tilde\beta}{\tilde\al}.
\end{equation} 
Note that $\frac{n-2-nm}{m(1-m)}<\frac{n-2}{m}.$  According to Lemma \ref{lem-v-derivative}, 
\begin{equation*}
\tilde  q'(r)=\frac{\tilde \alpha}{\tilde\beta}r^{\frac{\tilde\alpha}{\tilde\beta}-1}\left\{g(r)+\tilde k r g'(r)\right\}>0,\quad \forall r>0.
\end{equation*}
By direct computation, 
\begin{equation}\label{eq-q-q'}
\left(\frac{\tilde q'}{\tilde q}\right)'+ \frac{1}{r}\left(n-1-\frac{2m\tilde\alpha}{\tilde\beta}\right)\frac{\tilde q'}{\tilde q}+m\left(\frac{\tilde q'}{\tilde q}\right)^2+\frac{\tilde \beta }{m}\,r^{\frac{n-2-nm}{m}-(1-m)\frac{\tilde\alpha}{\tilde\beta}-1}\frac{\tilde q'}{\tilde q^m}=\frac{\tilde\alpha}{\tilde\beta}\cdot\frac{n-2-(m/\tilde k)}{r^2}.
\end{equation}
Let $$
h_2(r)=\exp\left( \frac{\tilde \beta}{m}\int_1^r \rho^{\frac{n-2-nm}{m}-(1-m)\frac{\tilde\alpha}{\tilde\beta}-1} \tilde q^{1-m}(\rho)\,d\rho\right)\quad\forall r>1.
$$
Then $h'_2(r)=\frac{\tilde\beta}{m} r^{\frac{n-2-nm}{m}-(1-m)\frac{\tilde\alpha}{\tilde\beta}-1 }\tilde q^{1-m}(r)h_2(r)$ and 
\begin{align*}
h_2(r)&\geq \exp\left( \frac{\tilde \beta\, \tilde q^{1-m}(1)}{m}\int_1^r \rho^{\frac{n-2-nm}{m}  -(1-m)\frac{\tilde\alpha}{\tilde\beta} -1}  d\rho\right)\\
&=\exp\left( \frac{\tilde \beta \, \tilde q^{1-m}(1)(r^{\frac{n-2-nm}{m}-(1-m)\frac{\tilde\alpha}{\tilde\beta} }-1) \,}{{n-2-nm}-m(1-m)  \frac{\tilde\alpha}{\tilde\beta}}\right) \quad\forall r>1
\end{align*}
which diverges   exponentially to infinity as  $r\to\infty$ since $\tilde \beta>0$ and 
$ \frac{\tilde\alpha}{\tilde\beta}<\frac{n-2-nm}{m(1-m)}.$ 
Let $c_1:= \tilde q^{m-1}(1) \tilde q'(1)h_2(1),$ and $c_2:=\frac{\tilde \alpha}{\tilde\beta}(n-2-(m/\tilde k)).$ Then $c_2>0.$  
Multiplying  \eqref{eq-q-q'} by $r^{n-1-\frac{2m\tilde\alpha}{\tilde\beta}} \,\tilde q^m(r)h_2(r)$ and integrating over $(1,r)$, we have
$$
r^{n-1-\frac{2m\tilde\alpha}{\tilde\beta}}\, \tilde q^m(r)h_2(r)\frac{\tilde q'(r)}{\tilde q(r)}= c_1+c_2\int_1^r\rho^{n-3-\frac{2m \tilde\alpha}{\tilde\beta}}\, \tilde q^m(\rho)h_2(\rho)d\rho,\quad\forall r>1.
$$
Let $p>0$ be a constant to be chosen later. By  the l'Hospital rule,  
\begin{align}
\limsup_{r\to\infty}r^{p}\frac{\tilde q'(r)}{\tilde q(r)} = &\limsup_{r\to\infty}\frac{c_1+c_2\int_1^r\rho^{n-3-\frac{2m \tilde\alpha}{\tilde\beta}} \tilde q^m(\rho)h_2(\rho)d\rho}{r^{n-p-1-\frac{2m\tilde\alpha}{\tilde\beta}} \tilde q^m(r)h_2(r)}\notag\\
\leq&\limsup_{r\to\infty}\frac{ c_2r^{n-3-\frac{2m \tilde\alpha}{\tilde\beta}} \tilde q^m(r)h_2(r)}{F(r)}\label{eq-tidle-q'-ineqn}
\end{align}
where 
\begin{align}
F(r)=&\frac{d}{dr }\left\{r^{n-p-1-\frac{2m\tilde\alpha}{\tilde\beta}} \tilde q^m(r)h_2(r)\right\}\notag\\
=&\left(n-p-1- \frac{2m\tilde\alpha}{\tilde\beta}\right)r^{n-p-2-\frac{2m\tilde\alpha}{\tilde\beta}} \tilde q^m(r)h_2(r) + mr^{n-p-1-\frac{2m\tilde\alpha}{\tilde\beta}} \tilde q^{m-1}(r) \tilde q'(r)h_2(r)\notag\\
&\qquad+r^{n-p-1-\frac{2m\tilde\alpha}{\tilde\beta}} \tilde q^m(r)h'_2(r)\notag\\
\geq&\left(n-p-1-\frac{2m\tilde\alpha}{\tilde\beta}\right)r^{n-p-2-\frac{2m\tilde\alpha}{\tilde\beta}} \tilde q^m(r)h_2(r) +\frac{\tilde\beta}{m}r^{n-p-2+\frac{n-2-nm}{m}-(1+m)\frac{\tilde\alpha}{\tilde\beta}} \tilde q(r)h_2(r).\label{eq-f-lower-bd}
\end{align}
Let $c_0:= \left(n-p-1- \frac{2m\tilde\alpha}{\tilde\beta}\right).$ By \eqref{eq-tidle-q'-ineqn} and \eqref{eq-f-lower-bd}, 
\begin{align}
0\leq \limsup_{r\to\infty}\frac{ r^{p} \tilde q'(r)}{\tilde q(r)}&\leq \limsup_{r\to\infty}\frac{c_2}{c_0r^{1-p}+(\tilde\beta/m)r^{1-p+\frac{n-2-nm}{m}  -(1-m)\frac{\tilde\alpha}{\tilde\beta}  } \tilde q^{1-m}(r)}\notag\\
&\leq \limsup_{r\to\infty}\frac{c_2}{c_0r^{1-p}+(\tilde\beta/m)r^{1-p+\frac{n-2-nm}{m}  -(1-m)\frac{\tilde\alpha}{\tilde\beta} } \tilde q^{1-m}(1)}. 
\label{eq-limt-rq_r/q} 
\end{align}
Since $ \frac{\tilde\alpha}{\tilde\beta}< \frac{n-2-nm}{m\,(1-m) }$,  
$$
1+\frac{n-2-nm}{m}-(1-m)\frac{\tilde\alpha}{\tilde\beta} >1.
$$ 
Hence it follows from  \eqref{eq-limt-rq_r/q} that 
\begin{equation}\label{eq-r-q_r-limit}
\lim_{r\to\infty}\frac{r^{p} \tilde q'(r)}{\tilde q(r)} =0\quad \forall  1<p<1+ \frac{n-2-nm}{m} -(1-m)\frac{\tilde\alpha}{\tilde\beta}.
\end{equation}
Let $p_0:=1+\frac{1}{2}\left(\frac{n-2-nm}{m}-(1-m)\frac{\tilde\alpha}{\tilde\beta}\right)$. Then by \eqref{eq-r-q_r-limit},
$$
|\log \tilde q(r)-\log \tilde q(1)|\leq C_1\int_1^r \rho^{-p_0}d\rho\leq C_2,\quad\forall r>1
$$
for some constants $C_1>0$, $C_2>0. $ Hence
$$
\tilde q(1)\leq  \tilde q(r)\leq e^{C_2}\, \tilde q(1),\quad\forall r\geq1.
$$ 
Then the monotonicity of $\tilde q$    implies that  $\displaystyle\lim_{r\to\infty} \tilde q(r)=A$ for some constant $A>0$ and the proposition follows.  
\end{proof}

\begin{cor}\label{cor-behavior-infinity-inversion}
 Let $n\geq3,$ $0< m <\frac{n-2}{n} ,$    $\tilde \alpha>0,\tilde\beta>0 $,  $\frac{\tilde\alpha}{\tilde\beta}<\frac{n-2-nm}{m\,(1-m)}$ and $\eta>0.$  
 Let $g $ be the solution of  \eqref{eq-fde-inversion}  in $(0,\infty)$ given by Theorem \ref{thm-existence-inversion} which satisfies 
\begin{equation}\label{eq-g-class}
\left\{\begin{aligned}
&g\in C^1([0,\infty);\R)\cap C^2((0,\infty);\R)\quad\quad&\mbox{ if }\quad 0<m<\frac{n-2}{n+1}\\
&g\in C^{0, \delta_0}([0,\infty);\R)\cap C^2((0,\infty);\R)\quad&\mbox{ if }\quad\frac{n-2}{n+1}\leq m<\frac{n-2}{n}\end{aligned}\right. 
\end{equation}
and \begin{equation}\label{eq-g(0)-g'(0)}
\left\{\begin{aligned}
&g(0)=\eta,\quad   g_r(0)=0 &\quad\mbox{ if  } \quad 0<m<\frac{n-2}{n+1}\\
&g(0)=\eta,\quad \lim_{r\to 0^+}r^{\delta_1} g_r(r)=-\frac{\tilde\alpha \eta^{2-m} }{n-2-2m}&\quad\mbox{if } \quad \frac{n-2}{n+1}\leq m<\frac{n-2}{n} 
\end{aligned}\right.
\end{equation}
 where $\delta_0$ and $\delta_1$ are given by  \eqref{defn-delta0-1}.  Then there exists a constant $A>0$ such that \eqref{eq-fde-scaled-at-origin-inv} holds.
\end{cor}

 \begin{proof} 
The result  follows from Proposition \ref{prop-behavior-infinity-inversion} since \eqref{eq-g(0)-g'(0)} implies  \eqref{eq-initial}. 
 \end{proof}
\section{Singular self-similar profiles}\label{sec-self-similar profiles} 
\setcounter{equation}{0}
\setcounter{thm}{0}

In this section we will use the inversion formula \eqref{eq-def-g-from-f} to prove the existence of radially symmetric solution of \eqref{eq-fde-scaled} in $\R^n\setminus\{0\}$ which has singular behavior at the origin of the form \eqref{eq-sol-behavior-origin} and decreases to zero at infinity. 

\begin{lemma}[Existence] \label{lem-existence-self-similar}
Let  $n\geq3,$ $0< m < \frac{n-2}{n} $. Suppose $\alpha$, $\beta$, $\rho_1$ satisfy \eqref{eq-alpha-beta-neg-condition}.  
Then for any $A>0$ there exists a  radially symmetric  solution $f$ of \eqref{eq-fde-scaled}   in $\R^n\setminus\{0\}$ 
which satisfies  \eqref{eq-sol-behavior-origin} and   \eqref{eq-sol-behavior-infty} for some   constant  $D_A>0$  depending on $A$. Moreover \eqref{eq-hess-f^m} holds.
\end{lemma}

\begin{proof}
Let  $\tilde \alpha$ and $\tilde\beta$ be given by  \eqref{eq-tilde-alpha-beta}.  Then \eqref{eq-tilde-alpha-beta-0}   and \eqref {eq-tilde-alpha-beta-1} hold.
By Theorem \ref{thm-existence-inversion}, there exists a unique radially symmetric  solution $g$  of \eqref{eq-inversion} in  $\R^n\setminus\{0\}$ which satisfies \eqref{eq-g-class} and \eqref{eq-g(0)-g'(0)} 
with $\eta=1$, where $r=|x|$ and $\delta_0, \delta_1$ are  given by \eqref{defn-delta0-1}.
It follows from  Corollary \ref{cor-behavior-infinity-inversion}  that 
\begin{equation}\label{eq-g-infty}
\lim_{r\to\infty} r^{\frac{\tilde\alpha}{\tilde\beta}} g(r)=A_0
\end{equation}
for some  constant $A_0>0.$ Let
$$
f(r):=r^{-\frac{n-2}{m}}g(r^{-1}),\quad r=|x|>0.
$$ 
By \eqref{eq-inversion}, \eqref{eq-g(0)-g'(0)} and \eqref{eq-g-infty}, $f$ is a radially symmetric solution to \eqref{eq-fde-scaled} in $\R^n\setminus\{0\}$ 
which satisfies  
\begin{equation*}
\lim_{r\to 0^+} r^{\frac{\alpha}{\beta}}f(r)= A_0,
\quad\mbox{and }\quad \lim_{r\to\infty} r^{\frac{n-2}{m}}f(r)= 1. 
\end{equation*} 
For any $\ld>0,$   let  
\begin{equation*}
\tilde{f}_{\ld}(x):= \ld^{\frac{2}{1-m}} f(\ld x).
\end{equation*}
Then $\tilde{f}_\ld$  satisfies \eqref{eq-fde-scaled} in $\R^n\setminus\{0\}$ with 
\begin{equation*}
\left\{\begin{aligned}
\lim_{r\to 0^+} r^{\frac{\alpha}{\beta}}\tilde{f}_\ld(r)&= \lim_{r\to 0^+} \ld^{\frac{2}{1-m}-\frac{\alpha}{\beta}}(\ld r)^{\frac{\alpha}{\beta}}f(\ld r)=\ld^{\frac{2}{1-m}-\frac{\alpha}{\beta}}A_0\\
\lim_{r\to\infty} r^{\frac{n-2}{m}}\tilde{f}_\ld(r)&= \lim_{r\to\infty} \ld^{\frac{2}{1-m}-\frac{n-2}{m}}(\ld r)^{\frac{n-2}{m}}f(\ld r)=\ld^{\frac{2}{1-m}-\frac{n-2}{m}} . 
\end{aligned}\right.
\end{equation*} 
For a given  $A>0,$ let $\ld:=(A/A_0)^{1/{(\frac{2}{1-m}-\frac{\alpha}{\beta})}}$.  Then $\tilde{f}_{\ld}$  satisfies \eqref{eq-sol-behavior-origin} and \eqref{eq-sol-behavior-infty} with $D_A=(A/A_0)^{(\frac{2}{1-m}-\frac{n-2}{m})/(\frac{2}{1-m}-\frac{\alpha}{\beta})}  $.  
By Lemma \ref{lem-v-derivative}, 
\begin{align*}
\alpha f(r)+ \beta r f_r(r)&=  \be \left\{ \frac{\alpha}{\beta}f(r)+ rf_r(r)\right\}\\
&=  \be  r^{-\frac{n-2}{m}} \left\{\left(\frac{\alpha}{\beta}-\frac{n-2}{m}\right) g(r^{-1})-  r^{-1} g' (r^{-1}) \right\}\\
&> \be  r^{-\frac{n-2}{m}} \left\{\left(\frac{\alpha}{\beta}-\frac{n-2}{m}\right) g(r^{-1}) +\frac{\tilde \al}{\tilde\be}g(r^{-1}) \right\}=0,\quad\forall r>0.
\end{align*} 
Hence  $\tilde{f}_{\ld}$ satisfies \eqref{eq-hess-f^m} and the lemma follows. 
\end{proof}
 
\begin{lemma}\label{lem-existence-self-similar-hess-sign}
Let  $n\geq3,$ $0<m<\frac{n-2}{n}$. Suppose $\alpha$, $\beta$, $\rho_1$ satisfy \eqref{eq-alpha-beta-neg-condition}.  
Let $f$ be a   radially symmetric  solution  of \eqref{eq-fde-scaled}  in     $\R^n\setminus\{0\}$ 
 satisfying
\begin{equation}\label{eq-sol-behavior-infty-equiv}
\lim_{|x|\to\infty}|x|^{\frac{n-2}{m}}f(x)= \eta 
\end{equation}
for some constant $\eta>0. $ Then $f$ satisfies \eqref{eq-hess-f^m}. 
\end{lemma}
\begin{proof}
Let $q(r):=r^{ \frac{ \alpha}{  \beta}}f(r)$, where $r=|x|>0.$ Then 
\begin{equation}\label{eq-q'-expression}
q'(r)=\frac{  \alpha}{ \beta} r^{\frac{ \alpha}{ \beta}-1}\left\{ f(r)+\frac{\be}{\al} r f'(r)\right\}\quad \forall r>0.
\end{equation} 
A direct computation implies 
\begin{equation}\label{eq-q-q'-equiv}
\begin{aligned}
&\left(\frac{q'}{q}\right)'+ \frac{1}{r}\left({n-1- \frac{2m \alpha}{ \beta}}\right)\frac{q'}{q}+m\left(\frac{q'}{q}\right)^2+\frac{  \beta \,}{m}r^{1  -(1-m) \frac{ \alpha}{ \beta} }\frac{q'}{q^m}=\frac{ \alpha}{ \beta}\cdot\frac{n-2- (m/k)}{r^2}
\end{aligned}
\end{equation}
where $k:=\be/\al.$ 
For any $0<\ve<1,$ let 
$$
h(r)=\exp\left( \frac{  \beta}{m}\int_\ve^r\rho^{1-(1-m) \frac{ \alpha}{ \beta} } q^{1-m}(\rho)d\rho\right)\quad\forall r>\ve.
$$
Since $\beta<0,$ $h(r)$ is a decreasing function of  $r>\ve$.  By   \eqref{eq-sol-behavior-infty-equiv}, there exists a constant $r_0>1$ such that 
\begin{align}
&\frac{\eta}{2}\leq r^{\frac{n-2}{m}}f(r)\leq 2\eta\qquad\qquad\,\,\,\forall r=|x|>r_0,\notag\\
\Rightarrow\quad&\frac{\eta}{2}r^{\frac{\al}{\be}-\frac{n-2}{m}}\leq q(r)\leq 2\eta r^{\frac{\al}{\be}-\frac{n-2}{m}}\quad\forall r>r_0.\label{eq-f-q-behavior-infty}
\end{align}
Hence
\begin{align}\label{eq-h-limit}
\lim_{r\to\infty }h(r)&\geq  \lim_{r\to\infty } \exp\left( \frac{  \beta }{m}\int_ \ve ^{r_0}\rho^{1-(1-m)\frac{ \alpha}{\beta}}q^{1-m}(\rho)d\rho+\frac{\beta}{m}(2\eta)^{1-m}\int_{r_0} ^{r} \rho^{1 - (1-m) \frac{ n-2}{ m} }  d\rho\right)\notag\\
&= C\exp\left( \frac{  \beta }{m}\int_ \ve ^{r_0}  \rho^{1  - (1-m) \frac{ \alpha}{ \beta}} q^{1-m}(\rho)d\rho\right)>0
\end{align}
for some constant $C>0$  since $2 - \frac{(1-m) ( n-2)}{ m}<0$. 

Let $c_1:=\ve^{n-1-\frac{2m\alpha}{ \beta}}q^{m-1}(\ve)q'(\ve),$ and $c_2:=\frac{  \alpha}{ \beta}\left(n-2-\frac{m}{  k}\right) .$ Note that $c_2>0.$  
We  multiply  \eqref{eq-q-q'-equiv} by $r^{n-1-\frac{ 2m\alpha}{\beta}}q^m(r)h(r)$ and integrate over $(\ve,r)$ to  have
\begin{equation}\label{defn-Q-eqn}
Q(r):=r^{n-1-\frac{2m \alpha}{ \beta}}q^m(r)h(r)\frac{q'(r)}{q(r)}=c_1+c_2\int_\ve^r\rho^{n-3-\frac{2m  \alpha}{ \beta}}q^m(\rho)h(\rho)d\rho\quad\forall r> \ve>0.
\end{equation}
 Since $h$ is positive for $r>\ve,$ and $\|h\|_{L^\infty([\ve,\infty])}<\infty,$ 
by \eqref{eq-f-q-behavior-infty}, 
\begin{align*}
\int_ \ve ^r\rho^{n-3-\frac{2m  \alpha}{ \beta}}q^m(\rho)h(\rho) d\rho\leq \int_ \ve ^{r_0}\rho^{n-3-\frac{2m  \alpha}{ \beta}}q^m(\rho)h(\rho)  d\rho+ C\int_ {r_0} ^\infty\rho^{-1-\frac{m  \alpha}{\beta} } d\rho<\infty 
\end{align*}
holds for any $r>\ve$ and  some constant $C>0.$   Hence  the  monotone increasing function $Q(r)$ is    bounded above in $(\ve,\infty)$. Thus   
$\lim_{r\to\infty} Q(r)$ exists. 

Now we claim that $\lim_{r\to\infty}Q(r) \leq 0.$ Suppose to the  contrary that $\displaystyle\lim_{r\to\infty}Q(r) >0.$    Then by \eqref{eq-sol-behavior-infty-equiv}, \eqref{eq-q'-expression} and \eqref{eq-h-limit}, 
\begin{align*}
0<\lim_{r\to\infty} r^{n-1- \frac{2m\alpha}{ \beta}}q^m(r) \frac{q'(r)}{q(r)}   = \eta^m\lim_{r\to\infty} r^{-  \frac{m\alpha}{ \beta}}\left\{   \frac{\al}{\be}+\frac {rf'(r)}{f(r)}\right\}=\eta^m \lim_{r\to\infty}  \frac {r^{1- \frac{m\alpha}{ \beta}} f'(r)}{f(r)}. 
\end{align*}
Thus  there exist constants $c_0>0$ and $r_1>1$ such that  
\begin{align*}
&\frac {f'(r)}{f(r)} \geq \frac{c_0}{\eta^m} r^{-1+  \frac{m\alpha}{ \beta}}\qquad\qquad\qquad\qquad\qquad\,\,\,\forall r>r_1\\
\Rightarrow\quad&\log f(r)\geq \log f(r_1) + \frac{c_0\be }{\eta^mm\al} \left( r^{\frac{m\alpha}{ \beta}}-r_1^{  \frac{m\alpha}{ \beta}}\right)\quad\forall r>r_1
\end{align*}
which contradicts   \eqref{eq-sol-behavior-infty-equiv}.  Hence  $\displaystyle\lim_{r\to\infty}Q(r) \leq 0.$
Since by \eqref{defn-Q-eqn} $Q(r)$ is a strictly  monotone increasing function of $r>\ve, $  it follows that   $Q(r)<0$  for any $r>\ve$. Thus   $q'(r)<0$ for any $r>0$ since $0<\ve<1$ is arbitrary. This together with \eqref{eq-q'-expression} implies that $f$ satisfies \eqref{eq-hess-f^m}. 
\end{proof}
 
\begin{lemma}\label{lem-equivalnece-problems}
Let $n\geq3,$ $0< m < \frac{n-2}{n}$. Suppose $\alpha$, $\beta$, $\rho_1$ satisfy \eqref{eq-alpha-beta-neg-condition}. Let $f$ be a radially symmetric  solution  of \eqref{eq-fde-scaled} in $\R^n\setminus\{0\}$ satisfying \eqref{eq-sol-behavior-origin} and \eqref{eq-sol-behavior-infty-equiv} for some constants  $A>0$ and $\eta>0. $  Let $g$, $\tilde\al$, $\tilde\beta$, $\delta_0$,  $\delta_1$ be given by \eqref{eq-def-g-from-f}, \eqref{eq-tilde-alpha-beta} and \eqref{defn-delta0-1} respectively. Then $g$  satisfies  \eqref{eq-inversion}  in $\R^n\setminus\{0\},$        \eqref{eq-fde-scaled-at-origin-inv},  \eqref{eq-g-class} and \eqref{eq-g(0)-g'(0)}. 
\end{lemma}
\begin{proof}
Let $g(0)=\lim_{r\to 0^+}g(r)$. Then by \eqref{eq-sol-behavior-infty-equiv} $g(0)=\eta$. By direct computation $g\in C([0,\infty);\R)\cap C^2((0,\infty);\R)$ satisfies \eqref{eq-inversion} in  $\R^n\setminus\{0\}$, and by  \eqref{eq-sol-behavior-origin},  \eqref{eq-fde-scaled-at-origin-inv} holds. We next prove that $rg'(r)\in L^{\infty}((0,\infty))$.  By \eqref{eq-tilde-alpha-beta-0} and Lemma \ref{lem-existence-self-similar-hess-sign}, 
\begin{align}\label{eq-hess-g-relation-f-g}
\frac{\tilde\al}{\tilde\beta}g(r)+  r g'(r)&=r^{-\frac{n-2}{m}}\left\{\left(\frac{\tilde\al}{\tilde\beta}-\frac{n-2}{m}\right) f(r^{-1})-    r^{-1}f'(r^{-1})\right\}\notag\\
&=-r^{-\frac{n-2}{m}}\left\{\frac{\al}{\beta}f(r^{-1})+r^{-1}f'(r^{-1})\right\}>0\qquad\forall r>0.
\end{align}
By \eqref{eq-inversion} and \eqref{eq-hess-g-relation-f-g},  
\begin{equation}\label{eq-hess-g-negative}
\frac{1}{r^{n-1}}\left( r ^{n-1}(g^m)'\right)'=-r^{\frac{n-2-nm}{m}-2}\left\{ {\tilde\al} g(r)+ \tilde\beta r g'(r)\right\}<0\quad\forall r>0. 
\end{equation}
Hence $r^{n-1}(g^m)'(r)$ is decreasing in $r>0.$ We now claim that 
\begin{equation}\label{eq-bound-rg_r-origin}
\lim_{r\to 0} r^{n-1}(g^m(r))' \leq 0. 
\end{equation} 
Suppose the claim is not true. Then there exists a constant $\delta>0$ such that 
\begin{equation*} 
r ^{n-1}(g^m)'(r) \geq \delta\quad\forall 0<r<\delta. 
\end{equation*}
Hence there exists a constant  $\tilde \delta >0$ such that $g'(r)\geq  \tilde \delta  \,r^{1-n}$ for $0<r<\delta$ since $g(0)=\eta>0.$ This implies that 
$$
g(\delta)-g(r)\geq \frac{\tilde \delta }{n-2}\left(r^{2-n}-\delta^{2-n} \right)\quad\forall 0<r<\delta,
$$ 
which  diverges to infinity as $r\to0.$ This is a contradiction. Hence the claim \eqref{eq-bound-rg_r-origin} holds.
Then by \eqref{eq-hess-g-negative} and  \eqref{eq-bound-rg_r-origin}, 
$$
r^{n-1}(g^m)'(r) < 0\quad\Rightarrow\quad g'(r)<0\quad \forall r>0.
$$ 
Hence by \eqref{eq-hess-g-relation-f-g},
\begin{equation}\label{eq-rg_r-bound}
-\frac{\tilde\al}{\tilde\beta}\,g(0)<-\frac{\tilde\al}{\tilde\beta}\,g(r)< r g'(r)< 0\quad \forall r>0.
\end{equation}
Thus $rg'(r)\in L^{\infty}((0,\infty))$. Now we claim that 
\begin{equation}\label{eq-zero-limt-rg_r-origin}
 \lim_{r\to 0^+}r g'(r)=0. 
\end{equation}
In order to prove \eqref{eq-zero-limt-rg_r-origin}, we note that by \eqref{eq-rg_r-bound}, 
$$
\lim_{r\to0}|r^{n-1}(g^m)'(r)|= mg^{m-1}(0) \lim_{r\to0}|r^{n-1}  g'(r)|\leq\frac{m\tilde\al g^{m}(0)}{\tilde\be}\cdot\lim_{r\to0}r^{n-2}=0 
$$ 
and hence by \eqref{eq-hess-g-negative},  
\begin{align}\label{eq-v-integral-form-equiv-0}
&r ^{n-1}(g^m(r))'=-\int_0^r\rho^{n-3+\frac{n-2-nm}{m}}\left\{\tilde\alpha g(\rho)+\tilde\beta \rho g'(\rho)\right\}d \rho\quad\forall r>0\notag\\
\Rightarrow\quad&g'(r)=-\frac{g^{1-m}(r)}{m r^{n-1}}\int_0^ r\rho^{n-3+\frac{n-2-nm}{m}}\left\{\tilde\alpha g(\rho)+\tilde\beta \rho g'(\rho)\right\}d\rho\quad\forall r>0.
\end{align}   
Thus by \eqref{eq-rg_r-bound} and \eqref{eq-v-integral-form-equiv-0}, 
\begin{equation*}
\begin{aligned}
\limsup_{r\to 0}\left|r g' (r)\right|&\leq \frac{2\tilde\al g^{2-m}(0)}{m }\lim_{r\to 0}\frac{\int_0^r\rho^{n-3+\frac{n-2-nm}{m}}\,d\rho}{r^{n-2}} =\frac{2\tilde\al g^{2-m}(0)}{n-2-2m}\,\lim_{r\to 0}   {r^{\frac{n-2-nm}{m}}}=0
\end{aligned}
\end{equation*}   
and \eqref{eq-zero-limt-rg_r-origin} follows.
 
Now we are ready to  prove   \eqref{eq-g(0)-g'(0)}. Since $g\in C([0,\infty);\R)$ and  $g(0)=\eta$, by \eqref{eq-zero-limt-rg_r-origin} for any $\ve>0$ there exists $\delta>0$ such that
\begin{equation}\label{eq-g'-l-infty-bd}
\eta/2\leq g(r)\leq 2\eta,\quad\mbox{and}\quad |r g'(r)|<\ve\quad \forall 0<r<\delta.
\end{equation}
Then by \eqref{eq-v-integral-form-equiv-0} and \eqref{eq-g'-l-infty-bd},
\begin{equation}\label{eq-v-integral-form-equiv} 
\left|g'(r)+ \frac{\tilde\alpha g^{1-m}(r)}{m r ^{n-1}} \int_0^ r \rho ^{ n-3+ \frac{n-2-nm}{m}}  g(\rho)d \rho\right|\leq    \frac{{\ve\,\tilde\beta \,(2\eta)^{1-m}} }{ n-2-2m} r^{  \frac{n-2-nm}{m}-1}\quad\forall 0<r<\delta.  
\end{equation}    
If $0<m<\frac{n-2}{n+1}$, \eqref{eq-v-integral-form-equiv} implies that
\begin{equation*}
\lim_{r\to 0}g'(r)=-\lim_{r\to 0}\frac{\tilde\alpha g^{1-m}(r)}{m r ^{n-1}}\int_0^r\rho^{n-3+\frac{n-2-nm}{m}}g(\rho)d \rho =  -\frac{\tilde\alpha \eta^{2-m}}{{n-2-2m}} \lim_{r\to 0} r^{ \frac{n-2-nm}{m} -1}=0.
\end{equation*} 
Then $g'$ can be extended to a continuous function on $[0,\infty)$ by setting $g'(0)=0$.  
If $\frac{n-2}{n+1}\leq m<\frac{n-2}{n},$ then by \eqref{eq-v-integral-form-equiv},
\begin{align*}
\lim_{r\to 0} r^{\delta_1}g'(r)=-\lim_{r\to 0}\frac{\tilde\alpha g^{1-m}(r)}{mr^{n-1-\delta_1}}\int_0^r\rho^{n-3+\frac{n-2-nm}{m}}  g(\rho)d \rho =-\frac{\tilde\alpha\eta^{2-m}}{{n- 2-2m}}
\end{align*}   
which implies that $r^{\delta_1}g'(r)$ can be extended to a continuous function $\tilde{h}$ (\mbox{say}) on $[0,\infty)$ by setting
$\tilde{h}(r)=r^{\delta_1}g'(r)$ for any $r>0$ and
$$
\tilde{h}(0)=-\frac{\tilde\alpha\eta^{2-m}}{{n- 2-2m}}.
$$
Then \eqref{eq-g(0)-g'(0)} holds and 
$$
g(r)=\eta+\int_0^r\rho^{-\delta_1}\tilde{h}(\rho)\,d\rho.
$$
Hence for any $s>0$, $0<r\le 1$, 
\begin{align*}
|g(r+s)-g(s)|=&\left|\int_s^{r+s}\rho^{-\delta_1}\tilde{h}(\rho)\,d\rho\right|\le C\left|\int_s^{r+s}\rho^{-\delta_1}\,d\rho\right|\\
\le&C\left|\left(r+s\right)^{1-\delta_1}-s^{1-\delta_1}\right|\\
\le&Cr^{1-\delta_1}\left|\left(1+\frac{s}{r}\right)^{1-\delta_1}-\left(\frac{s}{r}\right)^{1-\delta_1}\right|\\
\le&Cr^{1-\delta_1}\\
\le&Cr^{\delta_0}
\end{align*}
where $C>0$ is a generic constant.
Thus \eqref{eq-g-class} holds. 
\end{proof}

We are now ready for the proof of Theorem \ref{thm-existence-self-similar}.
\vspace{0.08in}

\noindent{\it Proof of Theorem \ref{thm-existence-self-similar}:} 
For any $A>0,$ existence of a radially symmetric solution $f$ of \eqref{eq-fde-scaled} in $\R^n\setminus\{0\}$, which satisfies \eqref{eq-sol-behavior-origin} and \eqref{eq-sol-behavior-infty} follows from Lemma  \ref{lem-existence-self-similar}. 
By Lemma \ref{lem-existence-self-similar-hess-sign} $f$ satisfies \eqref{eq-hess-f^m}.
In order to prove uniqueness, we let $g$, $\tilde\al$, and $\tilde\beta$ be given by \eqref{eq-def-g-from-f} and \eqref{eq-tilde-alpha-beta} respectively. By Lemma \ref{lem-equivalnece-problems}, $g$  satisfies  \eqref{eq-inversion}  in $\R^n\setminus\{0\},$ \eqref{eq-fde-scaled-at-origin-inv}, \eqref{eq-g-class}, and \eqref{eq-g(0)-g'(0)}.
Then Theorem \ref{thm-existence-inversion} yields that such $g$ is   unique. Hence  the uniqueness of  $f$ follows.      
\qed

\begin{cor}
Under the same assumption as Theorem \ref{thm-existence-self-similar}, for any $A>0,$ let $f$ be the unique radially symmetric solution of \eqref{eq-fde-scaled} in $\R^n\setminus\{0\}$ which satisfies \eqref{eq-sol-behavior-origin} and \eqref{eq-sol-behavior-infty} for some  constant  $D_A>0$  depending on $A.$ Then the following holds: 
\begin{enumerate}[(i)]
\item $\lim_{r \to 0^+}r^{ \frac{\al}{\be}+1} f'(r)=-\frac{\al}{\be}A$
\item $\lim_{r \to\infty}r^{\frac{n-2}{m}+1} f'(r)=-\frac{n-2}{m}D_A.$
\end{enumerate}
\end{cor}
\begin{proof}   
Let $g$, $\tilde\al$, $\tilde\beta$ be given by \eqref{eq-def-g-from-f} and \eqref{eq-tilde-alpha-beta} respectively.  
By Lemma \ref{lem-equivalnece-problems}, $g$  satisfies  \eqref{eq-inversion}  in $\R^n\setminus\{0\},$ \eqref{eq-fde-scaled-at-origin-inv}, \eqref{eq-g-class}, and \eqref{eq-g(0)-g'(0)}.
Let $\tilde q$ be given by \eqref{eq-def-tilde-q}. By \eqref{eq-tilde-alpha-beta-1} and the proof of  Proposition \ref{prop-behavior-infinity-inversion}, \eqref {eq-r-q_r-limit} holds since \eqref{eq-g(0)-g'(0)} implies  \eqref{eq-initial}.  Hence by \eqref{eq-fde-scaled-at-origin-inv}, \eqref{eq-r-q_r-limit},  and \eqref{eq-hess-g-relation-f-g}, for any $1<p<1+ \frac{n-2-nm}{m} -(1-m)\frac{\tilde\alpha}{\tilde\beta}$,
$$
0=\lim_{r\to\infty}\frac{r ^p \cdot r ^{ \frac{\tilde\alpha}{\tilde\beta}-1}\left\{\frac{\tilde \alpha}{\tilde\beta} g(r)+ r g'(r) \right\}}{r ^{ \frac{\tilde\alpha}{\tilde \beta}}g(r)}=-\lim_{r\to\infty}\frac{r ^{p-1}  \cdot r ^{-\frac{\al}{\be}} \left\{\frac{  \alpha}{ \beta} f(r ^{-1})+ r ^{-1} f'(r ^{-1}) \right\}}{A}.
$$  
Then it follows that    
$$
0=\lim_{r\to 0^+}r^{\frac{\al}{\be}} \left\{\frac{  \alpha}{ \beta} f(r)+ r f'(r) \right\}= \frac{\al}{\be}A+ \lim_{r \to0^+}r^{ \frac{\al}{\be}+1} f'(r). 
$$
By \eqref{eq-g(0)-g'(0)}, 
\begin{align*}
0=\lim_{r\to 0^+}  r g'(r)&=-\lim_{r\to0^+}\left\{ \frac{n-2}{m}  r^{-\frac{n-2}{m}}f(r^{-1})+   r^{-\frac{n-2}{m}-1} f'(r^{-1})\right\} \\&=-\frac{n-2}{m}D_A-\lim_{r \to\infty}r^{\frac{n-2}{m}+1} f'(r),
\end{align*} 
which finishes the proof.
\end{proof}
 
\begin{remark}[Monotonicity and Comparison]\label{rmk-scaling-monotonicity}
 Let $\alpha$, $\beta$ and $\rho_1$ satisfy \eqref{eq-alpha-beta-neg-condition} and 
  $f_{\lambda}$ be as in Remark \ref{self-similar-soln-scaling-property-rmk}. Then by Lemma \ref{lem-existence-self-similar-hess-sign} for any $r=|x|>0$,
\begin{align*}
\frac{d}{d\ld}f_{\ld}(r)&=\ld^{\frac{2}{1-m}-1}\left\{\frac{2}{1-m}   f_1(\ld r)+(\ld r) f_1'(\ld r)\right\}\\ 
&< \ld^{\frac{2}{1-m}-1}\left\{\frac{\al}{\be} f_1(\ld r)+(\ld r) f_1'(\ld r)\right\}\\
&<0.
\end{align*}
Hence for any $\ld_1>\ld_2>0$, there exists a constant $0<c_0<1$ such that 
$$ 
c_0 f_{\ld_2}(r)\leq f_{\ld_1}(r)<f_{\ld_2}(r)\quad\forall r>0
$$ 
since $\lim_{r\to 0} \frac{f_{\ld_1}(r)}{f_{\ld_2}(r)}>0 $ and $\lim_{r\to\infty} \frac{f_{\ld_1}(r)}{f_{\ld_2}(r)}>0.$
\end{remark}
 
\section{Existence and  asymptotic behavior of singular solutions } \label{sec-asymptotic-behavior}
\setcounter{equation}{0}
\setcounter{thm}{0}

Let  $n\geq3$ and  $0< m < \frac{n-2}{n} $.  Let  $\rho_1=1$ and  $\alpha$, $\beta$  satisfy \eqref{eq-alpha-beta-neg-condition}.   For any $\lambda>0$ we let $f_{\lambda}$ be as in Remark \ref{self-similar-soln-scaling-property-rmk} and $U_\ld $ be given by \eqref{eq-def-U-lambda} for the rest of the paper.   This section will be devoted to the study of singular solutions  of \eqref{eq-fde-global-except-0} 
trapped in between two self-similar  solutions $U_{\ld_1}, U_{\ld_2} $ of the form  \eqref{eq-def-U-lambda} for  some constants  $\ld_1>\ld_2>0.$   
For  our convergence, we will restrict ourselves to the case $\frac{2}{1-m}<\frac{\alpha}{\beta} < n ,$ which guarantees  the integrability of singular solutions of \eqref{eq-fde-global-except-0} near the origin. 
   
\subsection{Existence} 

We will first prove Theorem \ref{thm-weighted-L1-contraction-new} which is a weighted $L^1$-contraction principle  with weight $|x|^{-\mu}$ for $\mu\in (\mu_1,\mu_2)$.

\vspace{0.08in}

\noindent{\it Proof of Theorem \ref{thm-weighted-L1-contraction-new}:} 
We choose  $\eta\in C_0^\infty(\R^n)$   such that $0\leq \eta\leq1, $ $\eta=1$ for $|x|\leq 1,$ and $\eta=0$ for $|x|\geq 2.$ For $R>2,$ and $0<\ve<1,$ let $\eta_{R}(x):=\eta(x/R)$, $\eta_\ve(x):=\eta(x/\ve)$, and $\eta_{\ve,R}(x)=\eta_R(x)-\eta_\ve(x).$ Then   $|\D \eta_{\ve,R}|^2 +|\La \eta_{\ve,R}|\leq C\ve^{-2}$   for $\ve\leq |x|\leq 2\ve, $  and  $|\D \eta_{\ve,R}|^2 +|\La \eta_{\ve,R}|\leq CR^{-2}$   for $R\leq |x|\leq 2R.  $   
By Kato's inequality \cite{K}, 
\begin{equation}\label{eq-Kato-u-v-new}
\frac{\p}{\p t}|u-v|\leq \La|u^m-v^m|\quad\mbox{in $\sD'\left(\left(\R^n\setminus\{0\}\right)\times(0,\infty)\right)$.}
\end{equation}
Multiplying  \eqref{eq-Kato-u-v-new}  by $\eta_{\ve,R}(x) |x|^{-\mu}$  and integrating over $\R^n\setminus\{0\}$, 
we have  
\begin{align*}
&\frac{d}{dt}\int_{\R^n}|u-v|(x,t)\,\eta_{\ve,R}(x) |x|^{-\mu} dx \leq \int_{\R^n}| u^m-v^m|(x,t)\, \La\left(\eta_{\ve,R}(x) |x|^{-\mu}\right)dx\\
=&\int_{\R^n} | u^m-v^m|(x,t)\left\{ |x|^{-\mu}  \La \eta_{\ve,R}+2 \D \eta_{\ve,R} \cdot\D |x|^{-\mu}+\eta_{\ve,R} \La |x|^{-\mu}\right\}  dx.
\end{align*}
Since $0<\mu<\mu_2<n-2,$
\begin{equation}\label{eq-sign-hess-f-theta-new}
\begin{aligned}
\La |x|^{-\mu}&=\mu \left\{ \mu -(n-2) \right\}|x|^{-\mu-2}< 0\quad\mbox{in $\R^n\setminus\{0\}$}.
\end{aligned}
\end{equation}
Hence   
\begin{align}\label{eq-weighted-L1-contrac-test-eq-a-new}
\frac{d}{dt}\int_{\R^n}|u-v|(x,t)\,\eta_{\ve,R} (x)|x|^{-\mu} dx\leq& CR^{-2-\mu}  \int_{B_{2R}\setminus B_R} a(x,t)\,| u-v|(x,t)\,dx\notag\\
&\qquad +  C\ve^{-2-\mu} \int_{B_{2\ve}\setminus B_\ve} a(x,t)\,| u-v|(x,t)\,dx
\end{align}
where 
$$
a(x,t):=\int_0^1 \frac{m ds }{ \{s u+(1-s)v\}^{1-m}} \leq m U_{\ld_1}^{\,m-1}(x,t),\quad \forall (x,t)\in \left( \R^n\setminus \{0\}\right)\times(0,\infty)
 $$ 
by \eqref{eq-trapped-particular-sols-new}. 
It follows from    \eqref{eq-trapped-particular-sols-new}  and \eqref{eq-weighted-L1-contrac-test-eq-a-new} that  for $R>2 ,$ and $0<\ve<1,$ 
\begin{align}\label{eq-w-L1-contraction-1-new}
&\frac{d}{dt}\int_{\R^n}|u-v|(x,t) \eta_{\ve,R}(x) |x|^{-\mu} dx\notag\\
\leq&  C R^{-2-\mu}  \int_{B_{2R}\setminus B_R} U_{\ld_1}^{m-1} U_{\ld_2}(x,t)\, dx+  C \ve^{-2-\mu } \int_{B_{2\ve}\setminus B_\ve} U_{\ld_1}^{m-1}U_{\ld_2}(x,t)\,dx.
\end{align}
By Remark \ref{rmk-scaling-monotonicity} for any $t>0$ and $r>0,$ 
\begin{align}\label{eq-w-L1-contraction-1-integration-0-T-new}
\int_0^{\,t}\int_{B_{2r}\setminus B_r} U_{\ld_1}^{m-1}U_{\ld_2}\,dx\,ds=&\int_0^{\,t}\int_{B_{2r}\setminus B_r} s^{-m\al}  f_{\ld_1}^{m-1}(s^{-\beta}x)f_{\ld_2}(s^{-\beta}x) \,dx\,ds\notag\\
=&\int_0^{\,t} s ^{-m\al+n\beta} \int_{B_{2 s ^{-\beta}r}\setminus B_{s ^{-\beta}r}}  f_{\ld_1}^{m-1}(y)f_{\ld_2}(y)  \,dy\,ds\notag\\
\leq& c_0^{m-1}\int_0^{\,t} s ^{-m\al+n\beta}\int_{B_{2 s ^{-\beta}r}\setminus B_{s ^{-\beta}r}}  f_{\ld_2}^{m}(y) \,dy\,ds
\end{align}
where $c_0>0$ is as given in Remark  \ref{rmk-scaling-monotonicity}. 
Since $\frac{\al}{\be}<\frac{n-2}{m}$,  by  \eqref{eq-sol-behavior-origin} and \eqref{eq-sol-behavior-infty}, there exists a constant $C>0$ such that $f_{\ld_2}(r)\leq C\min \left( r^{-\frac{\al}{\beta}}, r^{-\frac{n-2}{m}}\right)$ for  $r>0.$   
Hence by \eqref{eq-w-L1-contraction-1-integration-0-T-new},           
\begin{align} \label{eq-w-L1-contraction-2-new-pre}
&\int_0^{\,t}\int_{B_{2r}\setminus B_r} U_{\ld_1}^{m-1}U_{\ld_2} \,dx\,d s\notag\\
\leq& C \int_0^{\,t} s ^{-m\al+n\beta}\int_{B_{2 s ^{-\beta}r}\setminus B_{s ^{-\beta}r}} \min \left( |y|^{-\frac{m\al}{\beta}}, |y|^{- {(n-2)}{ }}\right)  dy\,d s\notag\\
\leq&C \int_0^{\,t} s ^{-m\al+n\beta} \min\left\{ ({s ^{-\beta}r})^{n-\frac{m\al}{\be}} ,({s ^{-\beta}r})^{2}  \right\}  d s\notag\\
\leq&\left \{\begin{aligned}
&C\,t\,r^{n-\frac{m\al}{\be}}  \qquad\qquad\qquad\qquad\qquad\qquad&\mbox{ if }t^{-\be} r\leq 1\\
&C r^{n-\frac{m\al}{\be}+\frac{1}{\be}} +Cr^2\int_{r^{1/\be}}^t s ^{(n-2)\beta -m\al }ds \quad&\mbox{ if }t^{-\be} r> 1. 
\end{aligned}\right. 
\end{align}
 Since $\alpha=\frac{2\beta-1}{1-m}$,  
$$
(n-2)\beta-m\alpha+1= n\be-\al. 
$$
If $t^{-\be}r>1,$ then
\begin{align*} 
\int_{r^{1/\be}}^t s ^{(n-2)\beta -m\al }\,ds\, 
=\,&\left \{\begin{aligned}
&\frac{1}{\al-n\be}\left(r^{n-\frac{\al}{\be}}-t^{n\be-\al}\right)       \qquad&\mbox{ if } n\be<\al  \\
& \log  (t\, r^{-\frac{1}{\be}}  ) \quad&\mbox{ if } n\be=\al   \notag\\
&\frac{1}{n\be-\al}\left( t^{n\be-\al} -r^{n-\frac{\al}{\be}}\right) \quad&\mbox{ if } n\be>\al  \notag
\end{aligned}\right.\notag
\end{align*}
and hence   by \eqref{eq-w-L1-contraction-2-new-pre},
\begin{align}\label{eq-w-L1-contraction-2-new}
&\int_0^{\,t}\int_{B_{2r}\setminus B_r} U_{\ld_1}^{m-1}U_{\ld_2} \,dx\,d s\notag\\
\leq&\left \{\begin{aligned}
&C\,t\,r^{n-\frac{m\al}{\be}}  \qquad\qquad\qquad\qquad\qquad\,&\mbox{if }t^{-\be} r\leq 1\\
&C r^{n+2-\frac{\al}{\be}} +Cr^2 \left\{ r^{\,n-\frac{\al}{\be}}\,+\,\log  (t\, r^{-\frac{1}{\be}}  )\,+\,t^{\,n\be-\al}\right\} \quad&\mbox{if }t^{-\be} r>1.
\end{aligned}\right.
\end{align}
By  \eqref{eq-w-L1-contraction-1-new}  and \eqref{eq-w-L1-contraction-2-new},     for any $t>0,$   
\begin{align}\label{eq-w-L1-contraction-3-new}
&\int_{\R^n}|u-v|(x,t)\, \eta_{\ve,R} (x)|x|^{-\mu} dx- \int_{\R^n}|u_0-v_0|(x) \eta_{\ve,R} (x)|x|^{-\mu} dx\notag \\
\leq &C_t\left(R^{n-\frac{\alpha}{\beta}-\mu}+R^{-\mu}\log R+R^{-\mu}+\ve^{n-2-\frac{m\al}{\be}-\mu}\right)
\end{align}
for sufficiently large   $R>2$   and   small $\ve\in (0,1)$.     
Letting $\ve\to0$ and $R\to\infty$ in \eqref{eq-w-L1-contraction-3-new}, \eqref{eq-weighte-L1-contraction-new}  
follows from  the choice of   $\mu  \in (\mu_1, \mu_2). $ 
By a similar argument as above,  \eqref{eq-weighte-L1-contraction-positive-part-new} holds. 
\qed 
    
\vspace{0.08in}
\noindent{\it Proof of Theorem \ref{thm-existence-sol-fde-singular}:} 
Note that $\al=\beta\gamma,$ $u_0\in L^\infty_{\loc}(\R^n\setminus\{0\}),$ and \eqref{eq-initial-trapped-particular-sols-existence} is equivalent to  
\begin{equation*}
U_{\ld_1}(x,0)\leq u_0(x)\leq U_{\ld_2}(x,0),\quad \forall x\in \R^n\setminus\{0\}.
\end{equation*}  
Uniqueness of solution of \eqref{eq-fde-global-except-0} satisfying \eqref{eq-existence-sols-trapped-self-similar}  then follows from  Theorem \ref{thm-weighted-L1-contraction-new}. We next observe that by \eqref{eq-behavior-f1_lambda}  $f_{\ld_i} $ satisfies  
\begin{equation}\label{eq-asymp-behavior-f-lambdas}
\lim_{|x|\to 0}|x|^{\frac{\al}{\be}}f_{\ld_i}(x)= A_{i}\quad\mbox{and}\quad
\lim_{|x|\to\infty}|x|^{\frac{n-2}{m}}f_{\ld_i} (x)= \overline D_{i} 
\end{equation} 
where $A_{i} = \ld_i^{\frac{2}{1-m}-\frac{\al}{\be}}$ and $\overline D_{i}  = \ld_i^{\frac{2}{1-m}-\frac{n-2}{m}}D_1>0$. 
By Theorem 2.2 of \cite{Hui2} combined with \eqref{eq-initial-trapped-particular-sols-existence} there exists a unique solution $u_R\in C(\overline{\cA_R}\times(0,\infty))\cap C^{\infty}(\cA_R\times(0,\infty))$ of
\begin{equation*}
\left\{\begin{aligned}
u_t=\La u^m\quad&\mbox{ in $ \cA_R \times(0,\infty),$}\\
u =U_{\ld_1}\quad&\mbox{ in $ \p \cA_R \times(0,\infty),$}\\
u(\cdot,0)=u_0\quad&\mbox{ in $ \cA_R,$}
\end{aligned}\right.
\end{equation*}
which satisfies  \eqref{eq-fde}  in $\cA_R\times(0,\infty)$ in the classical sense,  
$$ 
\|u_R(\cdot, t)- u_0\|_{L^1(\cA_R)}\to 0\quad\mbox{as $t\to0$,}
$$
and
\begin{equation}\label{eq-u_R-trapped-particular-sols-existence}
U_{\ld_1}\leq  u_R \leq U_{\ld_2}\quad\mbox{in $\cA_R\times(0,\infty)$.}
\end{equation} 
Since $f_{\lambda_1}$ satisfies \eqref{eq-hess-f^m},
\begin{equation}\label{eq-bdary-neg}
\p_t U_{\ld_1}=\Delta U_{\ld_1}^m=t^{-(m\al+2\beta)} \La f_{\ld_1}^m\,(t^{-\be} x)<0 ,\quad\forall (x,t)\in \p\cA_R\times(0,\infty).
\end{equation}
Hence by Theorem 2.2 of \cite{Hui2}, 
\begin{equation}\label{eq-AB-u-R}
\p_t u_R\leq \frac{u_R}{(1-m)t}\quad\mbox{in   $\cA_R\times(0,\infty)$}.
\end{equation}
For sake of completeness we will give a different simple proof of \eqref{eq-AB-u-R} here. Let $ v:=\frac{\p_t u_R}{u_R}$ and 
$$
\cP[z]:=\frac{m}{u_R} \La\left( u_R^m z\right)- \p_tz=m\left\{u_R^{m-1}\La z +\frac{2m}{m-1}\D u_R^{m-1}\cdot\D z +vz\right\}- \p_tz .
$$ 
By direct computation,  $v$ satisfies 
\begin{equation*}
\cP[v]=v^2\quad\mbox{in   $\cA_R\times(0,\infty)$.}
\end{equation*}
For any $\delta\in(0,1),$ we choose $\ve_\delta\in(0,\delta)$ such that 
$$   
v(\cdot,\delta) <\frac{1+\delta}{(1-m)(\delta-\ve_\delta)} \quad\mbox{in $\cA_R,$} 
$$
and define 
$$ 
w_\delta(t)=\frac{1+\delta}{(1-m)(t-\ve_\delta)}\quad\forall t\ge\delta.
$$
Then $w_\delta(t)$ satisfies 
$$
\cP[w_\delta]=mvw_\delta+\frac{1-m}{1+\delta} w_\delta^2\quad\mbox{in 
$\cA_R\times(\delta, \infty).$}
$$  
We claim that 
\begin{equation}\label{eq-AB-claim}
v(x,t)<w_\delta(t)\quad\mbox{in $\cA_R\times(\delta,\infty).$}
\end{equation}
By \eqref{eq-bdary-neg}, $v<0<w_\delta$ on  $\p\cA_R\times[\delta,\infty)$. Since $v<w_\delta$ on $\cA_R\times \{\delta\}$, if \eqref{eq-AB-claim} does not hold, then  there exists an interior point $(x_0,t_0)\in \cA_R\times(\delta,\infty)$ such that 
\begin{equation}\label{eq-AB-claim-comparison}
\left\{\begin{aligned}
v(x_0,t_0)=w_\delta(x_0,t_0)>0,\qquad &\D v(x_0,t_0)=\D w_\delta(x_0,t_0)\\
D^2v(x_0,t_0)\leq D^2w_\delta(x_0,t_0),\qquad &\p_t v(x_0,t_0)\geq \p_tw_\delta(x_0,t_0).
\end{aligned}\right.
\end{equation}
Then by \eqref{eq-AB-claim-comparison}, we deduce that at the point $(x_0,t_0),$
\begin{align*}
0<w_\delta^2=v^2= \cP[v]\leq  \cP[w_\delta]= \left( m +\frac{1-m}{1+\delta}\right)w_\delta^2<w_\delta^2
\end{align*}
which is a contradiction. Thus \eqref{eq-AB-claim} holds. Letting $\delta\to 0$ in \eqref{eq-AB-claim},
\eqref{eq-AB-u-R} follows since $\ve_\delta\in(0,\delta).$

Let $\Omega $ be a bounded open  subset  in  $ \R^n\setminus\{0\} $ such that $\overline\Omega\subset\R^n\setminus\{0\}.$ Then there is a bounded open subset $\tilde\Omega\subset\R^n\setminus\{0\}$ which contains $\overline\Omega$. By \eqref{eq-u_R-trapped-particular-sols-existence},  for any $0<\delta<T,$ the equation  for the sequence $\{u_R\}_{R>1}$ (for sufficiently  large  $R>1$)  is uniformly parabolic on $\tilde\Omega  \times(\delta/2,T]$. Hence  by the parabolic Schauder estimates \cite{LSU},  the sequence $\{u_R\}_{R>1}$ (for sufficiently  large  $R>1$ ) is equi-continuous in $C^{2,1}(\Omega\times (\delta, T])$. By the Ascoli Theorem and a diagonalization argument there exists  a sequence  $\{u_{R_i}\}_{i=1}^\infty$, $R_i\to\infty $ as $i\to\infty$,  such that $u_{R_i}$ converges to some function  $u\in C^{2,1}\left(\left(\R^n\setminus\{0\}\right)\times(0,\infty)\right)$ uniformly  in $C^{2,1}(K)$   as $i\to \infty$ for any compact set $K\subset\left(\R^n\setminus\{0\}\right)\times(0,\infty)$. Then $u$ satisfies \eqref{eq-fde}  in $\left(\R^n\setminus\{0\}\right)\times(0,\infty)$ in the classical sense, and \eqref{eq-existence-sols-trapped-self-similar}, \eqref{eq-AB}   follow from \eqref{eq-u_R-trapped-particular-sols-existence} and \eqref{eq-AB-u-R}. 

Now  we will prove that $u$ has initial value $u_0$. It suffices to prove that $$ \|u(\cdot, t)- u_0\|_{L^1(B_r(x_0))}\to 0\quad \mbox{  as $t\to0$}$$ for any ball $B_r(x_0)$ such that   $B_{2r}(x_0)\subset \R^n\setminus \{0\}.$ Fix such a ball $B_r(x_0)$ with $r>0,$ and  let $R_0>0$ be a constant such that $B_{2r}(x_0)\subset \cA_{R_0}.$ By using the Kato inequality and an argument similar to the proof of \cite[Lemma 3.1]{HP} we deduce that for  $R,R'>2R_0$, 
\begin{align*}
\left(\int_{B_r(x_0)} |u_{R'}-u_R|(x,t)dx\right)^{1-m}&\leq \left(\int_{B_{2r}(x_0)} |u_{R'}-u_R|(x,0)dx\right)^{1-m}+ Cr^{n(1-m)-2}\,t\\
&=Cr^{n(1-m)-2}\,t,\quad\forall t>0
\end{align*}  
for some constant $C>0.$ 
Letting $R'=R_i\to\infty,$ 
\begin{align*}
\int_{B_r(x_0)} |u-u_R|(x,t)dx &\leq Cr^{n-\frac{2}{1-m}}\,t^{\frac{1}{1-m}},\quad \forall t>0.
\end{align*}
Hence   for any $R_i>2R_0,$
\begin{align*}
&\limsup_{t\to0} \int_{B_r(x_0)} |u(x,t)-u_0(x)|dx\\
\leq&\limsup_{t\to0}\left\{\int_{B_r(x_0)} |u-u_{R_i}|(x,t)dx
+  \int_{B_r(x_0)} |u_{R_i}(x,t)-u_0(x)|dx \right\}\\
\leq&\limsup_{t\to0} Cr^{n-\frac{2}{1-m}}\,t^{\frac{1}{1-m}}=0
\end{align*}
which completes the proof of the theorem. 
\qed 
  
\begin{cor}\label{cor-existence-sol-fde-singular}
Let  $n\geq3,$ $0< m < \frac{n-2}{n} $, and $\frac{2}{1-m} <\gamma< n.$  Then the solution $u$ of \eqref{eq-fde-global-except-0} given by Theorem \ref{thm-existence-sol-fde-singular} with initial value $u_0$ satisfying \eqref{eq-initial-trapped-particular-sols-existence} for some constants $A_2>A_1>0$ is a weak solution of \eqref{eq-fde-global}. 
\end{cor}
\begin{proof} 
Let $  \vp\in C^\infty_0\left(\R^n\times(0,\infty)\right)$ be such that $\supp\vp\subset\R^n\times(t_1,t_2)$ for some constants   $t_2>t_1>0$. For $0<\ve<1,$ let $\eta_\ve\in C_0^\infty(\R^n )$ be    as in the proof of Theorem \ref{thm-weighted-L1-contraction-new}, and let $\al$, $\be$ be given by \eqref{eq-def-alpha-beta-gamma}. Since 
$$
f_{\ld_2}(x)\leq C|x|^{-\frac{\al}{\beta}}\quad\mbox{ in }\R^n\setminus\{0\}
$$ 
for some constant $C>0,$  \eqref{eq-existence-sols-trapped-self-similar} implies  that  
$$
u(x,t)\leq U_{\ld_2}(x,t)\leq Ct^{-\alpha}|t^{-\beta}x|^{-\frac{\alpha}{\beta}}=C|x|^{-\frac{\al}{\be}}\quad\forall (x,t)\in \left(\R^n\setminus\{0\}\right)\times(0,\infty). 
$$      
Then for any $0<\ve<1,$  
\begin{align}\label{eq-general-sol-weak-sense-at-0} 
&\left|\iint_{\R^n\times(0,\infty)} \left\{u^m\La \vp -u \vp_t\right\}dx\,dt\right|\notag\\
=&\left| \iint_{\R^n\times(0,\infty)} \left\{u^m\La (\eta_\ve \vp) -u \eta_\ve \vp_t\right\}dx\,dt\right|\notag\\
\leq&C  \left\{ \iint_{\left( B_{2\ve}\setminus B_\ve\right)\times(t_1,t_2)}   \ve^{-2}u^m dx\,dt  + \iint_{B_{2\ve}\times(t_1,t_2)}(u^m+u)dx\,dt\right\}\notag\\ 
\leq&C (t_2-t_1) \left(\int_{B_{2\ve}\setminus B_\ve}\ve^{-2}|x|^{-\frac{m\al}{\be}}dx+\int_{B_{2\ve}}|x|^{-\frac{\al}{\be}}  dx\right)\notag\\ 
\leq&C  \left( \ve^{n-2-\frac{m\al}{\be}}+\ve^{n-\frac{\al}{\be}}\right) 
\end{align}
since $\frac{\al}{\be}=\gamma< n<\frac{n-2}{m}$. 
    Since $0<\ve<1$ is arbitrary,  letting $\ve\to0$ in \eqref{eq-general-sol-weak-sense-at-0}, we deduce that  $u$ solves  
     \eqref{eq-fde}     
   in $\R^n\times(0,\infty)$ in the distributional sense. 
  
 Now  we will prove that $u$ has   initial value $u_0$.    It suffices to prove that  for any $R>0,$ $ \|u(\cdot, t)- u_0\|_{L^1(B_R)}\to 0$  as $t\to0$. 
 For any  $0<\ve<R,$
\begin{align}\label{eq-u-u0-limit}
&\limsup_{t\to 0} \int_{B_R} |u(x,t)-u_0(x)|\,dx\notag\\
\leq&\limsup_{t\to0}\left\{ \int_{B_\ve} |u(x,t)-u_0(x)|\,dx+  \int_{B_R\setminus B_\ve} |u(x,t)-u_0(x)|\,dx\right\}\notag\\
\leq&\limsup_{t\to0}  \int_{B_\ve} C|x|^{-\frac{\al}{\be}}\,dx + \limsup_{t\to0}\int_{B_R\setminus B_\ve} |u(x,t)-u_0(x)|\,dx\notag\\
\leq &C { \ve^{n-\frac{\al}{\be}}} 
\end{align}
since  $u$ is a solution of \eqref{eq-fde-global-except-0}.  Letting $\ve\to 0$ in \eqref{eq-u-u0-limit},
\begin{align*}
\lim_{t\to0} \int_{B_R} |u(x,t)-u_0(x)|dx=0,\quad\forall R>0
\end{align*}
and the corollary follows.
\end{proof}

\subsection{Large time asymptotics} 

In this subsection we will investigate  the large time behavior of the solution $u$ of \eqref{eq-fde-global-except-0}  given  by Theorem  \ref{thm-existence-sol-fde-singular} with  initial value $u_0 $ which  satisfies  \eqref{eq-initial-trapped-particular-sols-existence} for some constants $A_2>A_1>0$.  We will assume that $n\geq3,$ $0< m<\frac{n-2}{n}$, $\frac{2}{1-m}<\gamma<n$, and $\al$, $\be$ be given by  \eqref{eq-def-alpha-beta-gamma} for the rest of the paper. Notice that  such $u_0$  is integrable near the origin and $u\in 
C\left( [0,\infty); L^1_{\loc}(\R^n)\right)\cap C\left( (0,\infty); L^1(\R^n)\right)$. 
 
For any solution $u$ of \eqref{eq-fde} in $\left(\R^n\setminus\{0\}\right)\times(0,\infty),$ let $\tilde u$ be  the  rescaled function defined by \eqref{def-fde-rescaled} for $\be<0$ and  $\al=\frac{2\be-1}{1-m}$. Then the rescaled function  $\tilde u$ satisfes \eqref{eq-fde-rescaled} 
in $\left(\R^n\setminus\{0\}\right)\times (-\infty,\infty)$ in the classical sense and $\tilde U_\ld(y,\tau)= f_\ld(y) $ for any  $(y,\tau)\in \left(\R^n\setminus\{0\}\right)\times (-\infty,\infty)$ and $\lambda>0$.  If $u$ 
 satisfies   \eqref{eq-existence-sols-trapped-self-similar},  then 
$$
f_{\ld_1}(y) \leq \tilde u(y,\tau )\leq f_{\ld_2}(y)\quad\forall (y,\tau)\in \left(\R^n\setminus\{0\}\right)\times(-\infty,\infty)
$$ 
and in this case by the same argument as the proof of Corollary \ref{cor-existence-sol-fde-singular},  $\tilde u$ is a weak solution of \eqref{eq-fde-rescaled} in $\R^n\times(-\infty,\infty)$ since $\frac{2}{1-m}<\gamma<n$. Note that $\tilde u(\cdot, 0)\equiv u(\cdot, 1).$
      
We will first prove a strong contraction principle with weight $|x|^{-\mu_1}$ for such rescaled solutions  where $\mu_1=n-\frac{\al}{\be}=n-\gamma>0.$ We point out that the following strong contraction principle does not hold for the difference $f_{\ld_2}-f_{\ld_1}$ of two self-similar profiles $f_{\ld_1}, f_{\ld_2}$ for $0<\ld_2<\ld_1$ 
 since by  \eqref{eq-asymp-behavior-f-lambdas}, $$  f_{\ld_2}-f_{\ld_1}\not \in     L^1(r^{-\mu_1}; \R^n) .$$

  \begin{lemma}[Strong contraction principle] \label{lem-strict-L1-contraction-new}
Let $n\geq3,$ $0< m<\frac{n-2}{n}$, $\beta<0,\alpha=\frac{2\beta-1}{1-m}$  and $\frac{2}{1-m}<\frac{ \alpha}{ \beta} <n .$   
Let $\tilde u$ and $\tilde v$ be  solutions of \eqref{eq-fde-rescaled} in $\left(\R^n\setminus\{0\}\right)\times(0,\infty)$ with   initial values $\tilde u_0$ and $\tilde v_0$, respectively, such that 
\begin{equation}\label{eq-sol-trapped-particular-sols-rescaled-new}
f_{\ld_1}\leq\tilde u, \tilde  v\leq  f_{\ld_2}\quad\mbox{in $\left(\R^n \setminus\{0\}\right)\times(0,\infty)$}
\end{equation}
for some constants  $\ld_1>\ld_2>0.$ Suppose that 
$$
0\not\equiv\tilde u_0 -\tilde v_0 \in L^1\left(r^{-\mu_1};\R^n\right).
$$ 
Then  
\begin{equation*}
\|\tilde u(\cdot, \tau) -\tilde v(\cdot,\tau)\|_{ L^1\left(r^{-\mu_1}; \R^n\right) } < \| \tilde u_0  -\tilde v_0\|_{ L^1\left(r^{-\mu_1}; \R^n\right)}\quad\forall \tau>0.
\end{equation*}
\end{lemma}
  \begin{proof}
Let $q:=|\tilde u-\tilde v|.$ By the  Kato inequality, 
\begin{equation}\label{eq-q-diff-rescaled-new}
q_\tau\leq \La (\tilde aq)+\beta\Div(yq)+(\al-n\be) q \quad\mbox{in $\sD'\left(\left(\R^n\setminus\{0\}\right)\times(0,\infty)\right)$} ,
\end{equation}
where 
\begin{equation}\label{eq-a-lower-bd-new}
m f_{\ld_2}^{m-1}(y)\leq \tilde a(y,\tau):=\int_0^1 \frac{m ds }{ \left\{s \tilde u+(1-s)\tilde v\right\}^{1-m}} \leq m f_{\ld_1}^{m-1}(y)
\quad\forall y\in\R^n\setminus\{0\}, \tau>0.
\end{equation}
For any $R>2$ and $0<\ve<1,$ let $\eta_{\ve,R}$ be    as in the proof of Theorem \ref{thm-weighted-L1-contraction-new}. 
Multiplying  \eqref{eq-q-diff-rescaled-new}  by $\eta_{\ve,R}(y) |y|^{-\mu_1}$ and  integrating by parts,  for any  $\tau>0,$ 
\begin{equation}\label{eq-q-diff-rescaled-integration-xt-new}
\begin{aligned}
&\int_{\R^n} q (y,\tau)\eta_{\ve,R} (y)|y|^{-\mu_1} dy-\int_{\R^n} q (y,0)\eta_{\ve,R} (y)|y|^{-\mu_1}dy\\
\leq&\int_0^\tau\int_{\R^n} \left\{\tilde a\La |y|^{-\mu_1} -\beta y\cdot\D |y|^{-\mu_1} +(\al-n\be) |y|^{-\mu_1} \right\} q \eta_{\ve,R}  \,dy\,ds\\
&\qquad  +   \int_0^\tau\int_{B_{2R}\setminus B_R}  \left\{\tilde a\La \eta_{\ve,R}  |y|^{-\mu_1}+2 \tilde a \D \eta_{\ve,R}\cdot\D  |y|^{-\mu_1} -\be y\cdot\D \eta_{\ve,R} |y|^{-\mu_1} \right\}q\,dy\,ds\\
&\qquad  +  \int_0^\tau \int_{B_{2\ve}\setminus B_\ve}  \left\{\tilde a\La \eta_{\ve,R}  |y|^{-\mu_1}
+2 \tilde a \D \eta_{\ve,R}\cdot\D  |y|^{-\mu_1} -\be y\cdot\D \eta_{\ve,R}  |y|^{-\mu_1}  \right\}q\,dy\,ds.
\end{aligned}
\end{equation} 
Since  $0<\mu_1<\mu_2<n-2,$ by \eqref{eq-sign-hess-f-theta-new},
\begin{equation}\label{eq-q-diff-rescaled-integration-xt-hessian-minus-sign-new}
\begin{aligned}
\tilde a\, \La |y|^{-\mu_1}-\beta y\cdot \D |y|^{-\mu_1}+(\al-n\beta) |y|^{-\mu_1}
&< \left\{\al +(\mu_1-n)\be\right\}|y|^{-\mu_1}= 0 \quad\mbox{in $\R^n\setminus\{0\}$}.
\end{aligned}
\end{equation}
By    \eqref{eq-sol-trapped-particular-sols-rescaled-new},    \eqref{eq-a-lower-bd-new} 
and  Remark \ref{rmk-scaling-monotonicity},  for any $\tau>0,$ and $R>2,$
\begin{align}\label{eq-q-diff-rescaled-integration-xt-R-infty-new}
& \left|   \int_0^\tau\int_{B_{2R}\setminus B_R}  \left\{\tilde a\La \eta_{\ve,R}  |y|^{-\mu_1}+2 \tilde a \D \eta_{\ve,R}\cdot\D  |y|^{-\mu_1} -\be y\cdot\D \eta_{\ve,R}  |y|^{-\mu_1} \right\}q\,dy\,ds\right|\notag\\
\leq&C \left( R^{-2-\mu_1}  \int_{B_{2R}\setminus B_R}  f_{\ld_2}^{m}\,dy+ R^{-\mu_1} \int_{B_{2R}\setminus B_R}  f_{\ld_2}  \, dy \right)\tau \notag\\
\leq&C \left( R^{-2-\mu_1}  \int_{B_{2R}\setminus B_R}  (|x|^{-\frac{n-2}{m}})^m\,dy+ R^{-\mu_1} \int_{B_{2R}\setminus B_R}|x|^{-\frac{n-2}{m}}\,dy \right)\tau\notag \\
\leq&C\left(  R^{-\mu_1} +R^{n-\frac{n-2}{m}-\mu_1}\right)\tau,
\end{align}
which converges to zero as $R\to\infty$, and        for  any   $\tau>0,$ and $0<\ve<1,$   
\begin{align}\label{eq-q-diff-rescaled-integration-xt-epsilon-0-new}
&\left|\int_0^\tau\int_{B_{2\ve}\setminus B_\ve}\left\{\tilde a\La \eta_{\ve,R} |y|^{-\mu_1}+2 \tilde a \D \eta_{\ve,R}\cdot\D  |y|^{-\mu_1} -\be y\cdot\D \eta_{\ve,R} |y|^{-\mu_1} \right\}q\,dy\,ds\right|\notag\\
\leq&C \left( \ve^{-2-\mu_1}  \int_{B_{2\ve}\setminus B_\ve}  f_{\ld_2}^{m}  dy +\ve^{-\mu_1} \int_{B_{2\ve}\setminus B_\ve}   f_{\ld_2}  dy \right)\tau\notag\\
\leq&C\left(  \ve^{n-2-\frac{m\al}{\be}-\mu_1} +\ve^{n-\frac{\al}{\be}-\mu_1}\right)\tau=C\left(  \ve^{ {1}/{|\beta|}} +1\right) \tau\leq C\tau. 
\end{align}
 Hence
    letting  $R\to\infty$ and $\ve\to0$ in \eqref{eq-q-diff-rescaled-integration-xt-new},   by \eqref{eq-q-diff-rescaled-integration-xt-hessian-minus-sign-new}, \eqref{eq-q-diff-rescaled-integration-xt-R-infty-new},  and \eqref{eq-q-diff-rescaled-integration-xt-epsilon-0-new},
\begin{align}\label{eq-weighted-l1-bd}
&\int_{\R^n} q(y,\tau)  |y|^{-\mu_1} dy- \int_{\R^n} q(y,0) |y|^{-\mu_1} dy \notag\\
\leq&\limsup_{R\to\infty,\ve\to0} \int_0^\tau\int_{\R^n} \left\{\tilde a\La |y|^{-\mu_1} -\beta y\cdot\D  |y|^{-\mu_1} +(\al-n\be)  |y|^{-\mu_1} \right\} q \eta_{\ve,R}\,  dy\,ds + C\tau\notag\\ 
\leq&C \tau. 
\end{align}
Since $\tilde u_0-\tilde v_0\in L^1( r^{-\mu_1};\R^n),$   by \eqref{eq-weighted-l1-bd}  $\tilde u(\cdot,\tau )-\tilde v(\cdot,\tau)\in L^1( r^{-\mu_1};\R^n)$ for any $\tau>0$ and
\begin{equation*}
\int_0^\tau\int_{\R^n} q(y,s)  |y|^{-\mu_1} dy\,ds\leq \tau\int_{\R^n} q(y,0)  |y|^{-\mu_1}dy+C\tau^2,\quad\forall \tau>0. 
\end{equation*} 
Then by  \eqref{eq-q-diff-rescaled-integration-xt-epsilon-0-new}, 
\begin{align}\label{eq-weighted-L1-strict-contraction-pre-near-origin-new}
&\left|\int_0^\tau\int_{B_{2\ve}\setminus B_\ve}  \left\{\tilde a\La \eta_{\ve,R} |y|^{-\mu_1} +2 \tilde a \D \eta_{\ve,R}\cdot\D    |y|^{-\mu_1} -\be y\cdot\D \eta_{\ve,R}   |y|^{-\mu_1}  \right\}q\,dy\,ds\right|\notag\\
\leq&C \left(  \ve^{ {1}/{|\be|}\,}\tau+ \int_0^\tau\int_{B_{2\ve}\setminus B_\ve}  q(y,s)  |y|^{-\mu_1}  \,dy\,ds \right)\notag\\
\to&0\qquad\qquad\qquad\mbox{ as }\ve\to 0.
\end{align}
Therefore, letting  $R\to\infty$ and $\ve\to0$ in \eqref{eq-q-diff-rescaled-integration-xt-new},  
by  \eqref{eq-q-diff-rescaled-integration-xt-hessian-minus-sign-new}, \eqref{eq-q-diff-rescaled-integration-xt-R-infty-new},    \eqref{eq-weighted-L1-strict-contraction-pre-near-origin-new},  
and the assumption that     $\tilde u_0-\tilde v_0\not\equiv 0$ on $\R^n\setminus\{0\}$,  we deduce  that for any $\tau>0,$
     \begin{align*}
& \int_{\R^n} q(y,\tau)  |y|^{-\mu_1}dy- \int_{\R^n} q(y,0)  |y|^{-\mu_1}dy\\
  &\leq    \limsup_{R\to\infty,\ve\to0} \int_0^\tau\int_{\R^n} \left\{\tilde  a\La  |y|^{-\mu_1}-\beta y\cdot\D  |y|^{-\mu_1} +(\al-n\be) |y|^{-\mu_1} \right\} q \eta_{\ve,R}\,  dy\,ds<0
    \end{align*}
which  finishes the proof of the lemma. 
\end{proof}

\begin{lemma}[cf. Lemma 1 of \cite{OR}]\label{lem-Osher-Ralston-new}
Let $n\geq3,$ $0< m<\frac{n-2}{n}$, $\beta<0,\alpha=\frac{2\beta-1}{1-m},$ and $\frac{2}{1-m}<\frac{\alpha}{\beta}<n.$ 
Let $\tilde u$, $\tilde v$ be solutions of \eqref{eq-fde-rescaled} in $\left(\R^n\setminus\{0\}\right)\times(0,\infty)$ with initial values  $\tilde u_0$ and $\tilde v_0$ respectively, which satisfy  \eqref{eq-sol-trapped-particular-sols-rescaled-new} for some constants  $\ld_1>\ld_2>0.$ Suppose that there exists a constant $\ld_0\in[\ld_2,\ld_1]$ such that
$$
\tilde u_0 - f_{\ld_0}\in  L^1\left( r^{-\mu_1} ; \R^n\right)  
$$ 
and  
\begin{equation}\label{eq-Osher-Ralston-cond}
\lim_{i\to \infty }\|\tilde u(\cdot, \tau_i)-\tilde v_0\|_{L^1\left( r^{-\mu_1} ;\R^n\right)}=0
\end{equation}  
for some sequence $\{\tau_i\}_{i=1}^\infty$ such that $\tau_i\to\infty$  as $i\to\infty.$ Then 
\begin{equation}\label{eq-OR-new-tilde-v0}
\| \tilde v_0 -f_{\ld_0}\|_{ L^1\left(r^{-\mu_1}  ; \R^n\right) } 
\leq   \| \tilde u_0  - f_{\ld_0}\|_{ L^1\left( r^{-\mu_1}  ; \R^n\right)}
\end{equation}
and 
\begin{equation}\label{eq-OR-new}
\|\tilde v(\cdot, \tau)-f_{\ld_0}\|_{L^1\left( r^{-\mu_1};\R^n\right)}=\|\tilde v_0-f_{\ld_0}\|_{L^1\left(r^{-\mu_1} ;\R^n\right)}\quad\forall \tau>0.
 \end{equation}  
 \end{lemma} 
\begin{proof}    
We will use a modification of the proof of \cite{OR} to prove this lemma.  
By  the proof of Lemma \ref{lem-strict-L1-contraction-new}  and  Fatou's lemma together with \eqref{eq-Osher-Ralston-cond},  
\begin{align*}
\| \tilde u(\cdot, \tau_j)  - f_{\ld_0}\|_{ L^1\left(r^{-\mu_1}  ; \R^n\right) }
&\leq  \| \tilde u_0  - f_{\ld_0}\|_{ L^1\left( r^{-\mu_1}  ; \R^n\right) } \qquad\forall j\in\N\notag\\
\Rightarrow\qquad   \| \tilde v_0 -f_{\ld_0}\|_{ L^1\left(r^{-\mu_1}  ; \R^n\right) } 
&\leq   \| \tilde u_0  - f_{\ld_0}\|_{ L^1\left( r^{-\mu_1}  ; \R^n\right) }\qquad\mbox{as $j\to\infty$}
\end{align*}
and \eqref{eq-OR-new-tilde-v0} holds.
Then    the proof of Lemma \ref{lem-strict-L1-contraction-new} implies 
\begin{equation}\label{eq-lem-OR-contraction-new} 
\|\tilde v(\cdot,\tau)-f_{\ld_0}\|_{L^1( r^{-\mu_1};\R^n)}\leq\|\tilde v_0-f_{\ld_0}\|_{L^1( r^{-\mu_1};\R^n)}\quad\forall \tau>0.
\end{equation}
By \eqref{eq-OR-new-tilde-v0} and the proof of Lemma \ref{lem-strict-L1-contraction-new},  we have that for any $i\in\N,$ 
\begin{align*} 
\|\tilde u(\cdot,\tau_i)-\tilde v_0\|_{L^1\left(r^{-\mu_1};\R^n\right)}&\leq\|\tilde u(\cdot,\tau_i)-f_{\ld_0}\|_{L^1\left(r^{-\mu_1};\R^n\right)}+\|\tilde v_0-f_{\ld_0}\|_{ L^1\left( r^{-\mu_1};\R^n\right) } \\
&\leq 2\|\tilde u_0-f_{\ld_0}\|_{ L^1\left( r^{-\mu_1};\R^n\right) },
\end{align*}
and hence for any $\tau>0$ and $i\in\N,$
\begin{equation}\label{eq-lem-OR-tilde-v0-new} 
\begin{aligned} 
\|\tilde v_0 -f_{\ld_0}\|_{ L^1\left( r^{-\mu_1}  ; \R^n\right) } 
&\leq\liminf_{j\to\infty}\| \tilde u(\cdot, \tau_j)  - f_{\ld_0}\|_{ L^1\left( r^{-\mu_1}  ; \R^n\right) }\\
&\leq\| \tilde u(\cdot,\tau+ \tau_i)  - f_{\ld_0}\|_{ L^1\left( r^{-\mu_1} ; \R^n\right) }\\
&\leq\| \tilde u(\cdot,\tau+ \tau_i)-\tilde v(\cdot,\tau) \|_{ L^1\left( r^{-\mu_1}; \R^n\right) }+\| \tilde v(\cdot,\tau) - f_{\ld_0}\|_{ L^1\left(r^{-\mu_1}  ; \R^n\right) }\\
&\leq  \| \tilde u(\cdot, \tau_i) -\tilde v_0 \|_{ L^1\left( r^{-\mu_1}  ; \R^n\right) }+\| \tilde v(\cdot,\tau) - f_{\ld_0}\|_{ L^1\left( r^{-\mu_1} ; \R^n\right) }. 
\end{aligned}
\end{equation}     
Letting $i\to \infty$ in \eqref{eq-lem-OR-tilde-v0-new},  
\begin{align*}
\| \tilde v_0 -f_{\ld_0}\|_{ L^1\left( r^{-\mu_1} ; \R^n\right) } \leq \| \tilde v(\cdot,\tau) - f_{\ld_0}\|_{ L^1\left( r^{-\mu_1} ; \R^n\right) },\quad\forall \tau>0,
\end{align*}
which together with \eqref{eq-lem-OR-contraction-new} implies  \eqref{eq-OR-new}.  
\end{proof} 

We are now ready to prove the  local uniform  convergence   of the rescaled function of the solution of  \eqref{eq-fde-global-except-0} to an eternal solution of   \eqref{eq-fde-rescaled} in $\left(\R^n\setminus\{0\}\right)\times(-\infty,\infty)$ as well as   
 convergence in the weighted  $L^1$-space with weight $|x|^{-\mu_1}$ as $\tau\to\infty$.

\begin{lemma} \label{lem-uniform-convergence-tilde-u-compact-sets-new}
Let $n\geq3,$ $0< m<\frac{n-2}{n}$, $\frac{2}{1-m}<\gamma<n$, and   let $\al$, $\be$ be given by \eqref{eq-def-alpha-beta-gamma}. Let $u_0$ satisfy \eqref{eq-initial-trapped-particular-sols-existence}  and \eqref{eq-sol-initial-value-similar-to-U-ld0-new}
for some constants $A_2\ge A_0\ge A_1>0$ and $\mu_1<\mu<\mu_2$, where $\mu_1$, $\mu_2$ are given by \eqref{eq-weighted-L1-contraction-mu}.  Let $u$ be the solution of \eqref{eq-fde-global-except-0} which satisfies \eqref{eq-existence-sols-trapped-self-similar},  where  $\ld_i=A_i^{1/ (\frac{2}{1-m}-\gamma )}$ for $i=1,2$, and let $\tilde u(y,\tau)$ be given by \eqref{def-fde-rescaled}.  Let $\{\tau_i\}_{i=1}^\infty$ be a sequence such that $\tau_i\to\infty$  as $i\to\infty$  and  
\begin{equation}\label{eq-def-scaled-time-shift}
\tilde u_i(\cdot, \tau):=\tilde u(\cdot, \tau_i+\tau)\quad\forall \tau\in\R.
\end{equation} 
Then there exists a subsequence of $\{\tilde u_i\}_{i=1}^\infty$, which we still denote by $\{\tilde u_i\}_{i=1}^\infty$,  and  an eternal solution $\tilde v$ of \eqref{eq-fde-rescaled} in $\left(\R^n\setminus\{0\}\right)\times(-\infty,\infty)$  such  that  $\tilde u_i$ converges to $\tilde v$ uniformly on every compact subset of $\left(\R^n\setminus\{0\}\right)\times(-\infty,\infty)$ as $i\to\infty.$  
Moreover 
\begin{equation}\label{eq-claim-tilde-u-0-u(1)-new}
\tilde u(\cdot,0)  -f_{\ld_0} \in L^1\left( r^{-\mu_1}; \R^n\right)
\end{equation}
where  $\ld_0:=A_0^{1/ (\frac{2}{1-m}-\gamma )}$ and 
\begin{equation}\label{eq-rescaled-sol-convergence-wL1-new}
\lim_{i\to\infty}\,\|\tilde u_i(\cdot, \tau)-\tilde v(\cdot,\tau)\|_{L^1\left( r^{-\mu_1};\R^n\right)}= 0 \quad\forall \tau\in\R.
\end{equation}
\end{lemma}
\begin{proof} 
 Since $\tilde u$  satisfies  \eqref{eq-fde-rescaled} and \eqref{eq-sol-trapped-particular-sols-rescaled-0-new} in $\left(\R^n\setminus\{0\}\right)\times(-\infty,\infty)$, the equation \eqref{eq-fde-rescaled} for $\tilde u_i$ is uniformly parabolic in $\cA_R\times(-\infty,\infty) $ for any $R>1$. Then by the parabolic Schauder estimates \cite{LSU}, the sequence $\{ \tilde u_i\}_{i=1}^\infty$ is equi-continuous  in $C^{2,1}(K)$ for any compact set $K\subset(\R^n\setminus\{0\})\times(-\infty,\infty).$ By the Ascoli Theorem and a diagonalization argument, there exists a subsequence of the sequence $\{\tilde u_i\}_{ i=1}^\infty$, which we still denote by $\{\tilde u_i\}_{i=1}^\infty$ and some function $\tilde v \in C^{2,1}\left(\left(\R^n\setminus\{0\}\right)\times(-\infty,\infty)\right)$ such that $\tilde u_i$ converges to $\tilde v $ uniformly in $C^{2,1}(K)$ as $i\to\infty$ for any compact set $K\subset(\R^n\setminus\{0\})\times(-\infty,\infty)$. Then $\tilde v$ is an eternal solution of \eqref{eq-fde-rescaled} in $\left(\R^n\setminus\{0\}\right)\times(-\infty,\infty)$  and  satisfies   
\begin{equation}\label{eq-sol-trapped-particular-sols-rescaled-0-new-tilde-v}
f_{\ld_1}   \leq   \tilde v \leq  f_{\ld_2} \quad\mbox{ in   {  $\left(\R^n\setminus\{0\}\right)\times(-\infty,\infty)$}.  }
\end{equation}  

We next observe that by \eqref{eq-existence-sols-trapped-self-similar},  \eqref{eq-sol-initial-value-similar-to-U-ld0-new},       
and Theorem \ref{thm-weighted-L1-contraction-new}, 
\begin{align*}
&\int_{ \R^n}|\tilde u(y,0)  -f_{\ld_0}(y) | |y|^{-\mu_1} dy\\ 
=&\int_{\R^n}|  u(y, 1)-U_{\ld_0}(y,1)  | |y|^{-\mu_1}dy\\
\leq&\int_{B_ 1}|  u(y, 1)-U_{\ld_0}(y,1) |  |y|^{-\mu}dy+   \int_{\R^n\setminus B_ 1}|  u(y, 1)-U_{\ld_0}(y,1) ||y|^{-\mu_1}dy\\
\leq&\int_{\R^n}|  u_0(y)-A_0|y|^{-\gamma} | |y|^{-\mu}dy+  2 \int_{\R^n\setminus B_ 1} f_{\ld_2}(y) |y|^{-\mu_1}   dy\\
\leq&\int_{\R^n}|  u_0(y)-A_0|y|^{-\gamma} | |y|^{-\mu}dy+  C \int_{1}^\infty  r^{n-1-\frac{n-2}{m}-\mu_1}  dr\\
\leq&\int_{\R^n}|  u_0(y)-A_0|y|^{-\gamma} ||y|^{-\mu}dy+  C' 
\end{align*} 
for some constants $C>0$, $C'>0$ and \eqref{eq-claim-tilde-u-0-u(1)-new} follows. Now we will prove \eqref{eq-rescaled-sol-convergence-wL1-new}. By  the proof of Lemma \ref{lem-strict-L1-contraction-new} and the Fatou Lemma,  
\begin{align}\label{eq-rescaled-sol-convergence-wL1-tilde-u_i-fatou-new} 
&\| \tilde u_i(\cdot, \tau) -f_{\ld_0}\|_{ L^1\left( r^{-\mu_1} ; \R^n\right) } \leq  \| \tilde u(\cdot,0)  - f_{\ld_0}\|_{ L^1\left( r^{-\mu_1} ; \R^n\right) }\quad\forall \tau  \geq -\tau_i\notag\\
\Rightarrow\quad&\|\tilde v(\cdot,\tau)-f_{\ld_0}\|_{L^1\left(r^{-\mu_1};\R^n\right) }\leq\|\tilde u(\cdot,0)-f_{\ld_0}\|_{ L^1\left( r^{-\mu_1};\R^n\right)}\quad\,\,\forall \tau   \in\R\quad\mbox{ as }i\to\infty.
\end{align}
For any $\tau\in\R$ and $R>1$,
\begin{equation}\label{eq-rescaled-convergene-split-new}
\begin{aligned}
&\int_{\R^n}|\tilde u_i(y, \tau)-\tilde v(y,\tau) ||y|^{-\mu_1}dy\\
\leq &\int_{\cA_R}|\tilde u_i(y, \tau)-\tilde v(y,\tau) ||y|^{-\mu_1}dy \\
&\qquad+\int_{\R^n\setminus \cA_R}|\tilde u_i(y, \tau)-f_{\ld_0}(y) ||y|^{-\mu_1}dy+  \int_{\R^n\setminus \cA_R}|\tilde v(y, \tau)-f_{\ld_0}(y) ||y|^{-\mu_1}dy.
\end{aligned}
\end{equation}
Let us fix $\tau\in\R,$ and let $\ve>0.$ By \eqref{eq-sol-trapped-particular-sols-rescaled-0-new} and  \eqref{eq-rescaled-sol-convergence-wL1-tilde-u_i-fatou-new}, there exists a constant  $R_1>1$ such that    for any $ R\geq R_1,$ 
\begin{equation}\label{eq-rescaled-convergene-split-outside-anulus-new-tilde-v}
\int_{\R^n\setminus \cA_R}|\tilde v(y, \tau)-f_{\ld_0} (y)||y|^{-\mu_1}dy\leq \ve,
\end{equation}
and 
\begin{align}\label{eq-rescaled-convergene-split-outside-anulus-new}
\int_{\R^n\setminus B_R}|\tilde u_i(y, \tau)-f_{\ld_0}(y) | \cdot |y|^{-\mu_1}dy 
\le&2 \int_{\R^n\setminus B_R} f_{\ld_2}(y)|y|^{-\mu_1} dy\leq C\int_{\R^n\setminus B_{R_1}}|y|^{-\frac{n-2}{m}-\mu_1} dy\notag\\
\le&CR_1^{-(\frac{n-2}{m}-\gamma)}\leq  \ve\qquad\forall i\in\N.
\end{align}
Let $t_i:= e^{\tau+\tau_i}$ for $i\in\N.$  Then by  \eqref{eq-sol-initial-value-similar-to-U-ld0-new}  and Theorem \ref{thm-weighted-L1-contraction-new},
\begin{align}\label{eq-rescaled-convergene-split-outside-before-scale-new}
\int_{B_{1/R_1}}|\tilde u_i(y, \tau)-f_{\ld_0}(y) ||y|^{-\mu_1}dy\leq&\int_{B_{1/R_1}}|\tilde u_i(y, \tau)-f_{\ld_0}(y) ||y|^{-\mu}dy\notag\\
=&t_i^{\,\al-n\be} \int_{B_{t_i^{\,\be}/R_1}}|u(x,  t_i)-U_{\ld_0} (x, t_i)|  |t_i^{-\be}x|^{-\mu}dx\notag\\
=&t_i^{\,\be(\mu-\mu_1)} \int_{B_{t_i^{\,\be}/R_1}}|    u(x,  t_i)-U_{\ld_0} (x, t_i)||x|^{-\mu}dx\notag\\
\le&t_i^{ \,\be (\mu-\mu_1)} \int_{\R^n}\left|    u_0(x)-A_0|x|^{-\gamma}  \right| |x|^{-\mu}dx.
\end{align}  
Thus by \eqref{eq-rescaled-convergene-split-new}, \eqref{eq-rescaled-convergene-split-outside-anulus-new-tilde-v},  
\eqref{eq-rescaled-convergene-split-outside-anulus-new}   and \eqref{eq-rescaled-convergene-split-outside-before-scale-new},    we deduce that  for any   $i\in\N,$
\begin{align}\label{eq-tilde-ui-tilde-v-l1-contraction}
&\int_{\R^n}|\tilde u_i(y, \tau)-\tilde v(y,\tau) | |y|^{-\mu_1} dy\notag\\
\leq&   \int_{\cA_{R_1}}|\tilde u_i(y, \tau)-\tilde v(y,\tau) | |y|^{-\mu_1}dy
+e^{ \,\be ({\mu}-{\mu_1})({\tau+\tau_i})} \int_{\R^n}\left|     u_0(x)-A_0|x|^{-\gamma}\right| |x|^{-\mu} dx +2 \ve. 
\end{align}
 Since $\be({\mu}-{\mu_1})<0,$ letting $i\to\infty$  in \eqref{eq-tilde-ui-tilde-v-l1-contraction}, 
 by the   uniform convergence of $\tilde u_i$ to $\tilde v$ on each compact subset of $\left(\R^n\setminus\{0\}\right)\times (-\infty,\infty)$, we obtain that  
  \begin{equation*}
 \begin{aligned}
 \limsup_{i\to\infty }\int_{\R^n}|\tilde u_i(y, \tau)-\tilde v(y,\tau) ||y|^{-\mu_1}dy& \leq  2\ve.
 \end{aligned}
 \end{equation*} 
 Since $\ve>0$ is arbitrary,  \eqref{eq-rescaled-sol-convergence-wL1-new} holds. 
 \end{proof}
 
\vspace{0.08in}

\noindent{\it Proof of Theorem \ref{thm-fde-rescaled-asymptotic}:}  
Let $\{\tau_i\}_{i=1}^\infty$ be any sequence such that $\tau_i\to\infty$  as $i\to\infty,$ and let  $\tilde u_i$  be given by \eqref{eq-def-scaled-time-shift}. By Lemma \ref{lem-uniform-convergence-tilde-u-compact-sets-new} there exists  a subsequence of the sequence $\{\tilde u_i\}_{i=1}^\infty$, which we still denote by $\{\tilde u_i\}_{i=1}^\infty$, that converges to an eternal solution  $\tilde v(y,\tau)$ of \eqref{eq-fde-rescaled} in $(\R^n\setminus\{0\})\times(-\infty,\infty)$ uniformly on  any compact subset of $\left(\R^n\setminus\{0\}\right)\times(-\infty,\infty)$ as $i\to\infty,$ and \eqref{eq-claim-tilde-u-0-u(1)-new} and  \eqref{eq-rescaled-sol-convergence-wL1-new} hold. 
    
Let $\tilde v_0(x)=\tilde v(x,0)$. Then by Lemma \ref{lem-Osher-Ralston-new}, \eqref{eq-OR-new-tilde-v0} and  \eqref{eq-OR-new} hold.  We   claim that $\tilde v_0 \equiv f_{\ld_0}$ in $\R^n\setminus\{0\}.$ Suppose to the contrary that $\tilde v_0\not\equiv f_{\ld_0}$ on $\R^n\setminus\{0\}$. Since $\tilde v$ satisfies \eqref{eq-sol-trapped-particular-sols-rescaled-0-new-tilde-v} with $\ld_i=  A_i^{1/ (\frac{2}{1-m}-\gamma )}$,  $i=1,2,$  by Lemma \ref{lem-strict-L1-contraction-new} together with \eqref{eq-OR-new-tilde-v0},
\begin{equation*}
\|\tilde v(\cdot,\tau)-f_{\ld_0}\|_{ L^1\left( r^{-\mu_1} ; \R^n\right) }< \| \tilde v(\cdot,0)-f_{\ld_0}\|_{ L^1\left( r^{-\mu_1} ;\R^n\right) }\quad\forall \tau>0
\end{equation*}
which contradicts \eqref{eq-OR-new}. Thus we conclude  that $\tilde v_0\equiv f_{\ld_0}$ in $\R^n\setminus\{0\},$ and    $\tilde u_i(\cdot,0)= \tilde u(\cdot ,\tau_i) $ converges to $f_{\ld_0} $ uniformly  on each compact subset  of $\R^n\setminus\{0\}$ as $i\to\infty.$

Since  the sequence $\{\tau_i\}_{i=1}^\infty$ is arbitrary, we deduce  that $\tilde u(\cdot,\tau)$  converges   to $f_{\ld_0}$  uniformly  on each compact subset of $\R^n\setminus\{0\}$ as $\tau\to\infty.$ By \eqref{eq-rescaled-sol-convergence-wL1-new}, 
$$
\lim_{\tau\to\infty}\,\|\tilde u(\cdot, \tau)-  f_{\ld_0}\|_{L^1\left( r^{-\mu_1}  ;\R^n\right)}= 0
$$
which completes the proof of the theorem.
\qed

\begin{remark} 
\begin{enumerate}[(a)]
\item
Under the same assumption as in Theorem \ref{thm-fde-rescaled-asymptotic}, if we restrict ourselves  to the case   
\begin{equation}\label{eq-gamma-restrict}
\max\left( \frac{2}{1-m}, \frac{n}{m+1}\right) < \gamma=\frac{\al}{\be} <  n,
\end{equation}
we can obtain results similar to Theorem \ref{thm-fde-rescaled-asymptotic}  using a  different weighted $L^1$-space.
More precisely, let $$\theta_1:=\frac{\be}{\al} \,\mu_1\quad\mbox{and}\quad\theta_2:=\frac{\be}{\al}\,\mu_2.$$ 
Then    \eqref{eq-gamma-restrict} implies $\theta_1< m.$ For any $\theta\in (0,m]\cap(\theta_1, \theta_2),$ consider  the weighted $L^1$-space with weight $f^{\theta}:=f_{\ld_2}^\theta$ defined by 
\begin{equation*}
L^1(f^\theta; \R^n):=\left\{ h: \int_{\R^n } |h(x)| f^\theta(x) dx<\infty\right\}
\end{equation*}
with  norm $$\|h\|_{L^1(f^\theta; \R^n)}=\int_{\R^n } |h(x)| f^\theta(x) dx.$$
Then $L^1\left(f^{\theta_1}; \R^n\right)$ is a slightly bigger space than  $L^1\left(r^{-\mu_1}; \R^n\right)$  since   by \eqref{eq-asymp-behavior-f-lambdas},
\begin{equation*}
f^{\theta_1}(x)=\left\{    
\begin{aligned}
A_2^{\,\theta_1}|x|^{-\mu_1}(1+o(1))\qquad &\mbox{as $|x|\to0$, }\\
 \overline D_{2}^{\,\theta_1}   |x|^{- \frac{n-2}{m}\cdot \frac{\be}{\al}\cdot \mu_1}(1+o(1))\qquad &\mbox{as $|x|\to\infty$. }
\end{aligned}\right.
\end{equation*} 
Replacing $L^1\left(r^{-\mu}; \R^n\right)$ for  $\mu\in  (\mu_1, \mu_2)$,  and \eqref{eq-sol-initial-value-similar-to-U-ld0-new} in Theroem \ref{thm-fde-rescaled-asymptotic} by  $L^1\left(f^\theta; \R^n \right)$ for  $\theta\in (0,m]\cap(\theta_1, \theta_2)$ and  
\begin{equation*}
u_0- A_{0} |x|^{-\gamma}\in L^1\left(f^\theta; \R^n \right)
\end{equation*}
one can  deduce  that the rescaled function $\tilde u(y,\tau)$ given by \eqref{def-fde-rescaled} converges  to  $f_{\ld_0}$ with $\ld_0:=A_0^{1/(\frac{2}{1-m}-\gamma)}$, as $\tau\to\infty,$ uniformly on  every compact subset of $\R^n\setminus\{0\} $, and in $L^1\left( f^{\theta_1}; \R^n\right)$ by using similar arguments as the proof of Theorem \ref{thm-fde-rescaled-asymptotic}. In fact  \eqref{eq-gamma-restrict} which implies  that  $\theta_1< m$ is needed in the proof of the corresponding strong contraction principle with weight $ f^{\theta_1}$ for rescaled solutions. More specifically,  for any $\tau>0,$      
\begin{align*}
&\tilde a(y,\tau) \La f^{\theta_1}-\beta y\cdot \D f^{\theta_1}+(\al-n\beta)f^{\theta_1}\\
\leq&m f^{m-1} \La f^{\theta_1}-\beta y\cdot \D f^{\theta_1}+(\al-n\beta)f^{\theta_1} \\
\leq &m f^{m-1} \frac{{\theta_1}}{m} f^{{\theta_1}-m}\La f^m-\beta y\cdot \D f^{\theta_1}+(\al-n\beta)f^{\theta_1} \\
=&f^{{\theta_1}-1}\left\{  -2{\theta_1}  \beta y\cdot \D f+( -{\theta_1}\al +\al-n\beta)f \right\}\\
<&f^{{\theta_1}}\left\{  2{\theta_1} \al +(-{\theta_1}\al +\al-n\beta)\right\}
=0\qquad\qquad\mbox{in $\R^n\setminus\{0\}$}
\end{align*}
by  \eqref{eq-hess-f^m} where $\tilde a(y,\tau)$ is given by \eqref{eq-a-lower-bd-new} (cf. \eqref{eq-q-diff-rescaled-integration-xt-hessian-minus-sign-new}). 
 
\item 
If $\frac{n-2}{n+2}\leq m<\frac{n-2}{n},$  then   $  \frac{2}{1-m}\geq \frac{n}{m+1}$   and hence    
\eqref{eq-gamma-restrict}     holds for $ \gamma=\frac{\al}{\be} \in\left(\frac{2}{1-m}, n\right)$.
 
\end{enumerate}
\end{remark}


\end{document}